\documentclass[11pt,a4paper]{article}

\usepackage[a4paper,margin=24mm]{geometry}
\usepackage{amsmath,amssymb,amsthm,mathtools,bm}
\usepackage{graphicx,booktabs,array,multirow}
\usepackage[numbers,sort&compress]{natbib}
\usepackage[colorlinks=true,allcolors=blue!55!black]{hyperref}
\usepackage[nameinlink,noabbrev]{cleveref}
\usepackage{microtype}
\usepackage{xcolor}
\usepackage{caption}
\usepackage{subcaption}
\usepackage{enumitem}
\usepackage{siunitx}
\usepackage[section]{placeins}
\usepackage{authblk}
\usepackage{fancyhdr}
\usepackage{indentfirst}

\graphicspath{{figures/}}
\setlist{nosep,leftmargin=2em}
\numberwithin{equation}{section}

\crefname{figure}{Fig.}{Figs.}
\Crefname{figure}{Figure}{Figures}
\crefname{equation}{Eq.}{Eqs.}
\Crefname{equation}{Equation}{Equations}
\crefname{theorem}{Theorem}{Theorems}
\Crefname{theorem}{Theorem}{Theorems}
\crefname{proposition}{Proposition}{Propositions}
\Crefname{proposition}{Proposition}{Propositions}
\crefname{lemma}{Lemma}{Lemmas}
\Crefname{lemma}{Lemma}{Lemmas}
\crefname{corollary}{Corollary}{Corollaries}
\Crefname{corollary}{Corollary}{Corollaries}

\newtheorem{theorem}{Theorem}[section]
\newtheorem{proposition}[theorem]{Proposition}
\newtheorem{lemma}[theorem]{Lemma}

\theoremstyle{definition}

\theoremstyle{remark}

\newcommand{\R}{\mathbb R}
\newcommand{\Qraw}{\mathcal Q_{\mathrm{raw}}}
\newcommand{\Gg}{\mathcal G_h}
\newcommand{\Rr}{\mathcal R_h}
\newcommand{\sym}{\operatorname{sym}}
\newcommand{\skw}{\operatorname{skew}}
\newcommand{\tr}{\operatorname{tr}}
\newcommand{\im}{\operatorname{im}}

\newcommand{\norm}[1]{\left\lVert#1\right\rVert}
\newcommand{\abs}[1]{\left\lvert#1\right\rvert}
\newcommand{\XM}{X_M}
\newcommand{\XW}{X_W^{\mathrm{rel}}}

\title{Checkerboard Shells: A Position-Only Thin-Shell Discretization with Scale-Compatible Completion}
\date{}

\author[1]{Junjie Song}
\author[2]{Xuanyu Wu}
\author[3]{Zhifeng Zhang}
\author[1]{Bingtao Hu}
\author[1]{Zhaoxi Hong}
\author[1]{Xiuju Song}
\author[1,2]{Yixiong Feng\thanks{Corresponding author.}}
\author[1]{Jianrong Tan}

\affil[1]{State Key Laboratory of Fluid Power \& Mechatronic Systems, Zhejiang University, Hangzhou, China}
\affil[2]{College of Mechanical Engineering, Guizhou University, Guiyang, China}
\affil[3]{Advanced Technology Institute, Zhejiang University, Hangzhou, China}

\begin{document}
\maketitle

\begin{abstract}
Checkerboard edge-midpoint geometry has been used in computer graphics and discrete differential geometry for discrete isometries, developable-surface design, and curvature modeling. Its defining geometric property is that the four edge midpoints of any spatial quadrilateral form an exactly planar Varignon parallelogram. Consequently, an unambiguous local plane, tangent frame, and normal can be recovered from nodal positions even when the underlying quadrilateral is warped. Building on this structure, we develop a position-only thin-shell discretization that introduces no independent director, rotation, or strain variables.

The connected edge-midpoint surface is taken as the physical discrete midsurface. The first fundamental form is evaluated directly on the planar B faces, while the second fundamental form is constructed at staggered W sites from variations of neighboring B-face normals, so that membrane and bending kinematics share the same Checkerboard geometric carrier. A variational-kernel analysis further shows that second-order consistency on smooth geometry does not preclude lattice-scale blind modes: after quotienting out the raw checkerboard gauge, the B metric misses one physical membrane direction, whereas the symmetric W curvature misses two independent curvature directions. We therefore introduce a quotient-minimal membrane compatibility coordinate $X_M$ and a reference-relative curvature compatibility coordinate $X_W^{\mathrm{rel}}$. The former is normalized by $\beta_M=8\mu$ to the isotropic trace-free shear channel and enters the $O(t)$ membrane sector; the latter uses surface-polar transport to preserve objectivity under finite rotations, with coefficient $\beta_W^0=\mu\gamma_W^0$ obtained by exact integration of an auxiliary-Q1 normal-gradient energy on the reference metric, and enters the $O(t^3)$ bending sector.

Analytically, the formulation admits an explicit midpoint quotient, a complete classification of the flat blind modes, rigid-motion objectivity, reference-state consistency, and an $O(h^2)$ near-isometry approximation result for aligned generalized cylinders. Numerically, the W-centered second form exhibits second-order convergence on smooth surfaces; $X_M$ removes the target membrane defect while leaving standard load-bearing responses essentially unchanged; and the finite-mesh displacement effect of $X_W^{\mathrm{rel}}$ remains sub-percent and decreases under refinement. A simply supported plate, the Scordelis--Lo roof, the MacNeal--Harder M3 shell, and three finite-deformation benchmarks further cover flat bending, membrane--bending coupling on curved shells, thickness sensitivity, nearly $360^\circ$ rotation, nonlinear pinching, and localized ovalization. The resulting discretization unifies the strictly planar midpoint geometry of Checkerboard surfaces, the physical quotient space, and the variational observability required by thin-shell mechanics within a single position-only framework.
\end{abstract}

\noindent\textbf{Keywords:} thin shells; position-only discretization; Checkerboard geometry; discrete fundamental forms; membrane locking; blind modes; compatibility completion

\section{Introduction}\label{sec:intro}

Quadrilateral meshes provide two distinguished parametric directions and are widely used in geometry processing, surface design, and fabrication. A generic spatial quadrilateral, however, is not planar. If its four vertices are treated directly as a local surface patch, tangent planes, normals, and curvatures typically depend on a choice of diagonal triangulation, local fitting, or additional rotational variables. Checkerboard edge-midpoint geometry offers a different route. By Varignon's theorem, the four edge midpoints of any quadrilateral form an exactly planar parallelogram. Thus, even when the raw quadrilateral is warped, the nodal positions alone determine an unambiguous local planar structure. Adjacent midpoint quadrilaterals share edge midpoints and thereby form a connected discrete surface with a natural directional structure.

This representation has developed into a substantial line of work in computer graphics and discrete differential geometry. Peng et al.\ studied the geometric constraints and design properties of Checkerboard patterns \cite{pengEtAl2019}. Jiang et al.\ used quad-mesh-based isometries to represent developable surfaces and applied discrete isometries to shape design, transformation, and fabrication \cite{jiangEtAl2020,jiangEtAl2021}. Ceballos Inza et al.\ further related developable quad meshes to contact element nets \cite{ceballosInzaEtAl2023}. Together with the classical theory of discrete isothermic surfaces \cite{bobenkoPinkall1996}, these studies show that midpoints, diagonal directions, and local coplanarity can support rich discrete-surface geometry. More recently, Dellinger developed discrete first and second fundamental forms, Christoffel duality, and discrete isothermic nets directly from Checkerboard patterns \cite{dellinger2024}. Checkerboard geometry is therefore not an ad hoc device introduced for the shell formulation considered here, but an established position-based surface representation with demonstrated capabilities for discrete isometries, developability, and curvature modeling.

These properties naturally motivate a position-only thin-shell formulation. Kirchhoff--Love/Koiter theory measures membrane deformation through the first fundamental form of the midsurface and bending through the second fundamental form, with the associated energies scaling as $t$ and $t^3$, respectively \cite{koiter1966,naghdi1972,ciarlet2005,sauerDuong2017}. If both geometric quantities can be evaluated robustly from the current nodal positions, finite rotations can in principle be handled without introducing an independent director or rotation parametrization. At the same time, the dominant bending response of a thin shell is close to a midsurface isometry. Any unnecessary membrane constraint is therefore amplified relative to the physical bending stiffness as $t\to0$, which can lead to membrane locking \cite{chapelleBathe2011,stolarskiBelytschko1982,arnoldBrezzi1997,hiemstraEtAl2023}. The simplicity of a position-only discretization is therefore useful only if both its membrane and curvature observations have the correct physical observability.

Existing shell formulations address geometry, continuity, and locking from several directions. MITC, mixed-interpolation, and assumed-strain methods alleviate constraint mismatch by reconstructing or projecting strains \cite{dvorkinBathe1984,batheDvorkin1985,miYu2021,parkStanley1986}; Hellan--Herrmann--Johnson-type methods introduce mixed bending-moment variables \cite{neunteufelSchoberl2019}; and isogeometric shells obtain thin-shell accuracy through highly continuous basis functions together with continuous assumed-strain or hybrid discretizations \cite{hughesEtAl2005,kiendlEtAl2009,casqueroMathews2023,sauerEtAl2024}. Geometry-based shells instead construct bending measures from subdivision surfaces, hinges, discrete shape operators, or midedge states \cite{cirakEtAl2000,grinspunEtAl2003,grinspunEtAl2006,chenEtAl2018,chenEtAl2026}. These approaches are well established, but they commonly require higher continuity, additional edge/director/rotation or mixed variables, or distinct geometric/interpolation carriers for membrane and bending. Such additional structures do not directly exploit the strictly planar local patches and purely positional parametrization already available in a Checkerboard surface.

Checkerboard geometry therefore suggests a different starting point for shell discretization. Varignon midpoint patches remain exactly planar for arbitrary current nodal configurations, so their tangent frames and normals do not depend on diagonal triangulation; neighboring patches are naturally connected through shared midpoints. If both fundamental forms are derived directly from this midpoint surface, membrane and bending can share the same discrete geometric carrier, while finite rotations require updating only the nodal positions rather than maintaining independent director or rotation states. Extending Checkerboard geometry from graphics/DDG to thin-shell mechanics is therefore motivated not merely by broadening an existing geometric representation, but by exploiting its strict planarity, position-only state, and unified local carrier to obtain a more compact shell discretization.

Moving from geometry to mechanics introduces a Checkerboard-specific difficulty. Convergence of geometric quantities on smooth surfaces does not guarantee that the discrete energy has the correct observability at the lattice scale. The raw nodal coordinates contain an alternating gauge that represents the same midpoint surface, while other lattice-scale deformations can genuinely change the midpoint surface yet remain invisible to a discrete metric or curvature observation. The former is representation redundancy; the latter produces zero-energy or anomalously soft physical mechanisms. The two must first be separated in a physical quotient space, after which only the information genuinely missing from the fundamental forms should be supplemented. On the membrane side, this supplementation must be deliberately restrained so that the near-isometric bending motions of thin shells are not converted into artificial membrane stiffness. The central issue is thus the interaction among representation, variational observability, and thin-shell scaling rather than any single geometric approximation in isolation.

Following this perspective, the connected edge-midpoint surface is used as the physical midsurface, strictly planar Varignon B faces carry the local metric, and variations of neighboring B-face normals define the second fundamental form at staggered W sites. Membrane and bending therefore remain tied to the same Checkerboard geometric carrier. The variational kernels of these fundamental-form observations are then analyzed on the midpoint quotient, and the missing information is completed according to the distinct $O(t)$ membrane and $O(t^3)$ bending scales. The resulting finite-configuration model uses nodal positions as its only independent unknowns.

The main contributions are:
\begin{enumerate}
  \item \textbf{A position-only shell geometry with the Varignon midpoint surface as a unified carrier.} The compatible edge-midpoint surface is defined as the physical discrete midsurface. Planar B faces directly carry the first fundamental form, while a W-centered second form built from neighboring B-face normals describes bending without introducing independent normal or rotational variables. A finite-domain closure consistent with the W stencil is also provided. Membrane and bending therefore act on the same Checkerboard surface.
  \item \textbf{Identification of blind modes in the physical quotient space and scale-compatible targeted completion.} Quotient analysis separates the raw checkerboard gauge from genuine physical deformation. The B metric is shown to possess one additional membrane blind direction, while the symmetric W curvature possesses two independent curvature blind directions. $X_M$ restores only the one-dimensional membrane defect, whereas $X_W^{\mathrm{rel}}$ recovers the missing normal-compatibility information and preserves finite-rotation objectivity through surface-polar transport.
  \item \textbf{A finite-configuration Koiter-type position-only energy with competitive accuracy per state.} The two completion terms enter the energy at the physically appropriate $O(t)$ and $O(t^3)$ scales, with coefficients fixed by material shear normalization and auxiliary-Q1 geometric integration. Classical linear-shell benchmarks show that competitive accuracy can be obtained with only $3V$ positional scalars, compared with formulations that introduce independent edge states. Thickness--mesh coupling studies and three complete nonlinear loading paths further validate near-isometric response, large rotations, and strongly nonlinear curved-shell deformation.
\end{enumerate}

Together, these contributions carry the original geometric advantages of Checkerboard surfaces into shell mechanics: a unified physical midsurface supplies both membrane and bending geometry, quotient analysis ensures that the energy acts only on physical deformation, and scale-compatible completion restores the missing mechanical observability while preserving the position-only state.
\section{Checkerboard Representation and Discrete Shell Geometry}\label{sec:geometry}

\subsection{Raw Nodes, Midpoint Surface, and the Physical Quotient Space}\label{sec:quotient}

The free variables of the Checkerboard shell are still the nodal positions of a structured grid, but these raw vertices are not themselves the discrete midsurface to which mechanical meaning is assigned. The geometric measurements are instead taken on the adjacent edge midpoints: the four edge midpoints of each raw quadrilateral form an exactly planar Varignon parallelogram, and neighboring midpoint patches connect through shared midpoint nodes to form a continuous Checkerboard surface. Figure~\ref{fig:geometry}(a) shows the raw quadrilateral, its edge midpoints, and the resulting physical midpoint surface. A direct consequence is that the local tangent plane remains uniquely defined even when the raw quadrilateral is warped, without requiring an auxiliary triangulation or geometric fit.

Before constructing the shell energy, it is necessary to specify which changes of the discrete variables correspond to genuine midsurface deformation. The map from raw vertices to edge midpoints is not one-to-one: an alternating nodal displacement can change the raw mesh while leaving every edge midpoint unchanged, and therefore does not change the physical midsurface. If stiffness kernels were analyzed directly in the raw coordinate space, this representation redundancy would be mixed with genuine zero-energy physical deformation. The purpose of the physical quotient space is precisely to separate the two, so that subsequent kernel analysis concerns only deformations that alter the midpoint surface.

Consider a connected structured quadrilateral grid with raw state
\begin{equation}
 q=\{q_{ij}\in\R^3\}_{i=0,j=0}^{n_x,n_y}.
\end{equation}
The horizontal and vertical edge midpoints are
\begin{equation}\label{eq:midpoints}
 m^x_{i+1/2,j}=\frac{q_{ij}+q_{i+1,j}}2,
 \qquad
 m^y_{i,j+1/2}=\frac{q_{ij}+q_{i,j+1}}2.
\end{equation}
Let the linear midpoint map be denoted by $Pq=(m^x,m^y)$. We regard $Pq$ as the physical discrete surface, while the raw nodal vector $q$ is only a compact parametrization of that surface.

\begin{theorem}[Checkerboard gauge of the midpoint map]\label{thm:gauge}
If the raw graph is connected and bipartite, then
\begin{equation}\label{eq:gauge-kernel}
 \ker P=\left\{g_{ij}=(-1)^{i+j}A:\ A\in\R^3\right\}.
\end{equation}
Hence the physical configuration space satisfies
\begin{equation}\label{eq:quotient}
 \Qraw/\ker P\cong\im P.
\end{equation}
\end{theorem}

Equation~\eqref{eq:gauge-kernel} shows that the entire grid contains exactly one three-dimensional black--white alternating degree of freedom that leaves all physical midpoints fixed, and that this redundancy does not grow under mesh refinement. Figure~\ref{fig:geometry}(b) shows two raw representatives related by this gauge: their raw vertices differ, but their midpoint surfaces are identical. A graph-theoretic proof and the closure conditions characterizing the compatible midpoint image are given in Appendix~\ref{app:quotient}.

This gauge is a redundancy of the parametrization and should not be constrained by material stiffness or a penalty. It can instead be removed by a quotient basis, projection, or an equivalent gauge slice. Consequently, $P\delta q\neq0$ becomes the basic criterion for physical deformation. Only a deformation satisfying this criterion and simultaneously belonging to the linearized kernel of a membrane or curvature observation constitutes a physical defect that the shell energy must address.

If a conventional nodal mesh is required for output, a smooth representative can be selected from the equivalence class $q+\ker P$. This reconstruction is used only for visualization and data exchange; it does not enter the energy, internal force, or tangent computation. Once the physical state has been identified, metric and curvature are constructed on the same midpoint surface: the former measures lengths and angles within a local patch, while the latter measures relative rotations between neighboring local planes.

\subsection{Varignon B-Face Geometry and the Discrete First Fundamental Form}\label{sec:first-form-geometry}

The first fundamental form measures local lengths and angles of the midsurface and therefore requires a well-defined local tangent plane. Such a plane is not unique for a generic spatial quadrilateral, but it is exact for every Varignon midpoint patch used here. This inscribed planar patch is denoted a B face, as shown in Fig.~\ref{fig:geometry}(c). Its planarity is not a small-deformation approximation; it is a geometric identity that holds for arbitrary current nodal positions.

Let the raw parameter spacings be $h_x,h_y$, and define the two Checkerboard diagonal directions as
\begin{equation}
 d_1=(h_x,h_y),\qquad d_2=(-h_x,h_y),\qquad
 \ell=\sqrt{h_x^2+h_y^2}.
\end{equation}
For a raw cell $q_{00},q_{10},q_{11},q_{01}$, define the normalized B tangents
\begin{equation}\label{eq:b-tangents}
 t_1^B=\frac{q_{11}-q_{00}}{\ell},\qquad
 t_2^B=\frac{q_{01}-q_{10}}{\ell},\qquad
 T_B=[t_1^B,t_2^B],
\end{equation}
and the unit normal
\begin{equation}\label{eq:b-normal}
 N_B=\frac{t_1^B\times t_2^B}{\norm{t_1^B\times t_2^B}}.
\end{equation}
As long as $t_1^B\times t_2^B\neq0$, the tangent frame and normal of the B face are uniquely determined by the current positions. Its Gram matrix
\begin{equation}\label{eq:first-form}
 a_B=T_B^TT_B
\end{equation}
is the discrete first fundamental form: the diagonal entries measure squared lengths along the two B directions, while the off-diagonal entry measures their relative angle. Thus $a_B$ directly records in-plane stretch and shear of the physical B patch without introducing a director or an independent strain variable.

After aligning directions and scales, $a_B$ is algebraically identical to the local first-form observation in Dellinger's Checkerboard geometry \cite{dellinger2024}. The B-face metric is therefore locally consistent with the established Checkerboard metric construction. For a sufficiently smooth regular immersion, a centered chord expansion yields
\begin{equation}\label{eq:first-form-consistency}
 a_B=a+O(h^2),
\end{equation}
so the continuous first fundamental form is approximated to second order.

Equation~\eqref{eq:first-form-consistency} controls the local geometric error for smooth fields, but it does not exclude unobservable lattice-scale deformations. The completeness of the membrane energy must be examined only after both fundamental forms have been defined. For bending, the remaining task is to measure relative rotations of neighboring local planes on the same midpoint surface.

\subsection{W-Centered Second Fundamental Form and Finite-Domain Closure}\label{sec:second-form-geometry}

Whereas the first fundamental form describes lengths and angles within a local patch, bending is associated with the relative rotation of neighboring tangent planes. A discrete second fundamental form must therefore specify where normals are placed, how neighboring normals are compared, and whether the resulting curvature observation remains tied to the chosen physical midsurface.

Checkerboard and discrete-isothermic-surface research has developed Gauss-map, discrete-curvature, and second-form constructions \cite{bobenkoPinkall1996,pengEtAl2019,dellinger2024}. Dellinger's discrete second fundamental form, for example, uses a vertex-/coordinate-cross-centered Gauss image as the normal carrier and differentiates it along Checkerboard directions \cite{dellinger2024}. That construction is well suited to discrete isothermic nets, Christoffel duality, and curvature geometry, but its direct use in the present shell discretization would create three structural mismatches. First, the membrane metric is already attached to strictly planar Varignon B faces, whereas a vertex-centered Gauss image constitutes a different local geometric carrier; membrane and bending would no longer be derived directly from the same discrete midsurface. Second, the raw-to-midpoint representation has an exact gauge, so a shell energy must be analyzed in the physical quotient space to determine the variational kernel of its curvature observation; existing geometric constructions were not designed for this mechanical observability question. Third, finite-configuration shell mechanics requires the current/reference comparison of curvature quantities to remain objective under rigid rotations and to enter consistently at the $O(t^3)$ bending scale, requirements that go beyond the definition of a geometric curvature measure itself.

We therefore introduce a W-centered second form that acts directly on Varignon B-face normals. The first and second fundamental forms, $a_B$ and $b_W$, then share the same physical Checkerboard surface: the former measures the intrinsic metric of an individual planar B face, while the latter measures relative rotations of neighboring B faces. The construction adds no independent normal or director, requires no auxiliary plane fit in the generally non-coplanar W neighborhood, and introduces no second curvature carrier. More importantly, metric, curvature, gauge, and the subsequent compatibility completion can all be analyzed in the same position-only state space. For a shell whose physical midsurface is the B-face midpoint surface, this unified carrier is the central advantage of the W-centered construction over a vertex-centered Gauss-image route.

Because each B face already provides a unique $N_B$, the relative rotation of adjacent B planes can be measured directly through differences of these normals. The interstitial location among four neighboring B faces is therefore the natural curvature-sampling point; we denote it a W site. Figure~\ref{fig:geometry}(d) shows the connected four-B patch $B_{NE},B_{NW},B_{SW},B_{SE}$ surrounding a W site. The W site serves only as a staggered location at which neighboring B-face rotations are sampled; it does not introduce an independent geometric face.

The natural midpoint polygon around a W site is generally not planar, so no additional plane is fitted there. The W-centered tangents are obtained directly by crossed averages of the diagonal B tangents:
\begin{equation}\label{eq:w-tangents}
 t_1^W=\frac12(t_1^{B,NE}+t_1^{B,SW}),\qquad
 t_2^W=\frac12(t_2^{B,NW}+t_2^{B,SE}),
 \qquad T_W=[t_1^W,t_2^W].
\end{equation}
For tensor contraction in the bending sector, we also use the local W metric
\begin{equation}\label{eq:w-metric}
 a_W=T_W^TT_W.
\end{equation}
This auxiliary metric is used only to express and contract W-centered curvature tensors; it is not introduced as an additional membrane-strain observation.

Relative rotations of neighboring B planes are measured by centered normal differences:
\begin{equation}\label{eq:normal-diff}
 D_1^hN=\frac{N_{NE}-N_{SW}}{\ell},\qquad
 D_2^hN=\frac{N_{NW}-N_{SE}}{\ell}.
\end{equation}
The unsymmetrized pairing and the symmetric tensor used in the shell bending energy are
\begin{equation}\label{eq:second-form}
 (\widetilde b_W)_{\alpha\beta}=-D_\alpha^hN\cdot t_\beta^W,
 \qquad
 b_W=\sym\widetilde b_W.
\end{equation}
The skew part $\skw\widetilde b_W$ does not enter the shell energy. Because all of these quantities are constructed from B diagonal differences and B-face normals, they are exactly invariant under the finite-amplitude gauge in \cref{eq:gauge-kernel}.

Special boundary treatment is required only when the W-centered stencil of the second fundamental form is truncated. The first fundamental form $a_B$ is evaluated directly on every existing B cell and requires no analogous extrapolation. For a boundary W site, tangent averages, B-normal differences, and $b_W$ are reconstructed with one-dimensional second-order moment formulas: the boundary-value weights are $(15/8,-5/4,3/8)$ and the boundary first-derivative weights are $(-2,3,-1)/h$; at corners, the one-dimensional rule is applied by tensor product in the two parametric directions. W quadrature uses the dual fractions $1,1/2,1/4$ for interior, edge, and corner sites, respectively. These reconstructions complete the geometric stencil only and remain independent of Dirichlet, symmetry, or natural mechanical boundary conditions.

\begin{proposition}[Smooth consistency]\label{prop:consistency}
Let $r$ be a sufficiently smooth regular immersion, let the structured parameter grids belong to a fixed shape-regular family, and assume the B/W frames remain uniformly nondegenerate. Then, at the corresponding B/W sampling sites,
\begin{equation}\label{eq:smooth-consistency}
 a_B=a+O(h^2),\qquad b_W=b+O(h^2),
\end{equation}
where $a$ and $b$ denote the continuous first and second fundamental-form matrices expressed in the corresponding normalized Checkerboard diagonal basis.
The $O(h^2)$ error in $a_B$ follows from the B-chord geometry. Maintaining the same local order for $b_W$ at edges and corners additionally relies on the one-sided value/derivative reconstruction and its tensor-product corner extension.
\end{proposition}

The proposition follows from centered chord expansions, variations of the normalized cross product, and moment exactness of the one-sided reconstruction; technical details are given in Appendix~\ref{app:boundary}. Thus both the membrane metric and the curvature observation are defined on the same midpoint Checkerboard surface and remain second-order consistent for smooth fields. Truncation error, however, says nothing about the observability of lattice-scale modes. Before these geometric quantities can be used as a shell energy, their linearized kernels must be checked for non-rigid deformations that actually move the physical midsurface.

\begin{figure}[t]
 \centering
 \includegraphics[width=\textwidth]{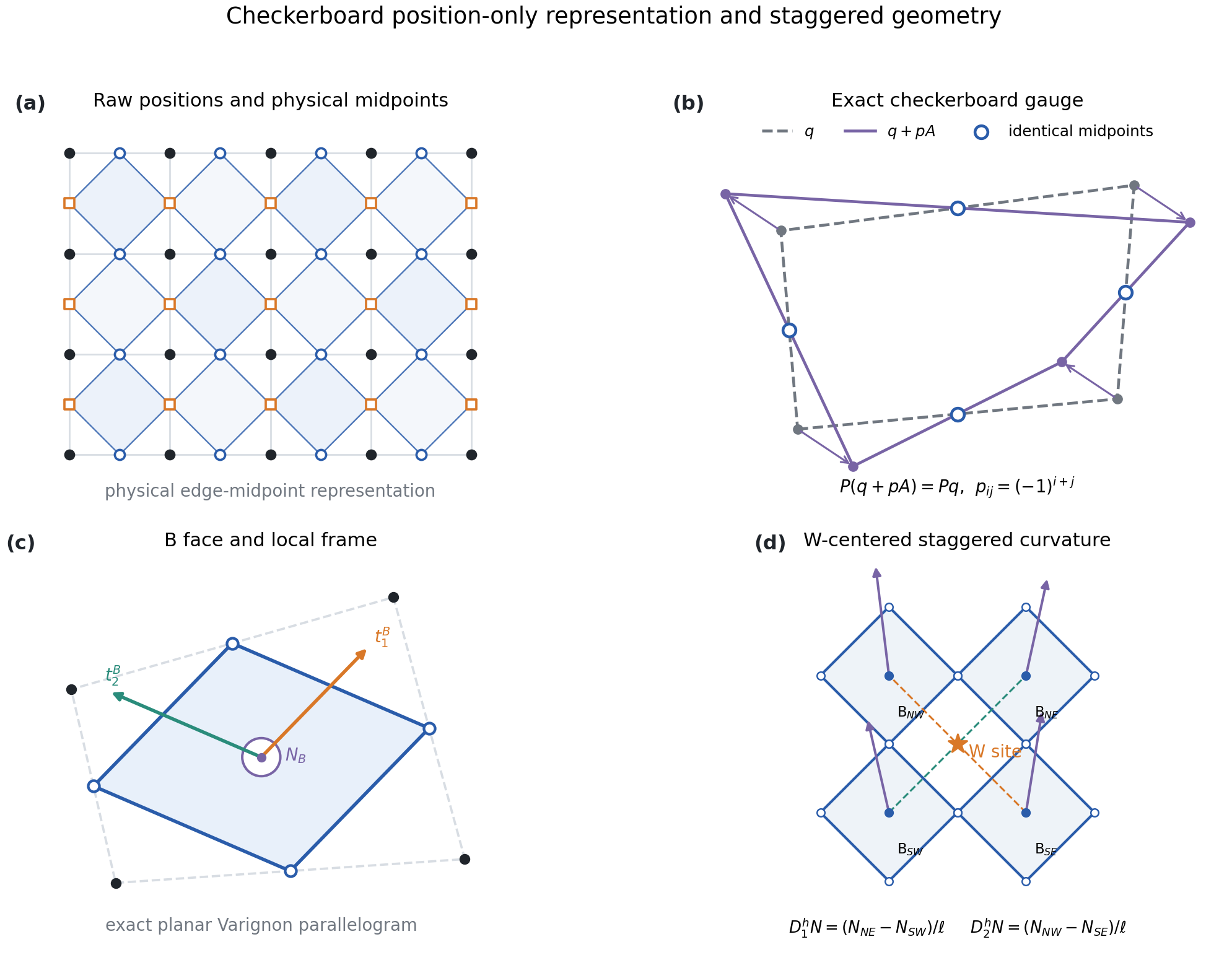}
 \caption{Checkerboard position-only representation and B/W geometry. (a) Raw vertices, edge midpoints, and the inscribed Varignon B faces obtained by connecting the four edge midpoints of each raw quadrilateral; neighboring B faces share midpoint nodes. (b) Two raw representatives related by the exact checkerboard gauge have the same midpoint surface. (c) A single planar B face with its tangent frame and normal. (d) A connected four-B patch surrounding one W site; diagonal centered differences of neighboring B normals define the W-centered curvature.}
 \label{fig:geometry}
\end{figure}
\FloatBarrier
\section{Physical Blind Modes and Variational Observability}\label{sec:blind}

\subsection{Geometric Consistency versus Variational Completeness}

A \emph{physical blind mode} is a genuine physical deformation that is not detected by a discrete geometric observation. In such a mode the physical midpoint surface changes, while the observation used to construct the energy remains unchanged to first order. If $\delta q$ is not a rigid motion and satisfies
\begin{equation}\label{eq:blind-mode-definition}
 P\delta q\neq0,\qquad D\mathcal A_h(q_0)[\delta q]=0,
\end{equation}
then the first condition excludes the representation gauge identified in Section~\ref{sec:quotient}, whereas the second states that this genuine midsurface deformation is invisible to the discrete observation. Since the material energy depends on the configuration only through such observations, a physical direction in this kernel lacks the corresponding quadratic restoring stiffness.

Blind-mode analysis is necessary because smooth consistency and lattice-scale observability are different properties. An $O(h^2)$ truncation estimate is derived from smooth Taylor fields and primarily probes deformations whose wavelength is much larger than the mesh scale. A structured lattice, however, also supports alternating modes with wavelength comparable to $h$. If a physical mode of this type lies in the kernel of the membrane or curvature observation, the discrete energy acquires an additional zero-energy or anomalously soft direction. Conversely, suppressing such underconstraint by adding overly broad membrane observations can shrink the near-isometric space of the continuous thin shell and induce locking. Kernel analysis therefore serves two purposes: it determines the dimension of the missing information and it identifies how narrowly a subsequent completion should act.

The analysis linearizes $a_B$ and $b_W$ separately about a flat reference state, removes the continuous rigid motions and the raw gauge of Section~\ref{sec:quotient}, solves the remaining kernel through shared-node lattice compatibility, and finally uses $P\delta q$ to determine whether each surviving direction actually moves the physical midsurface. For a generic discrete observation $\mathcal A_h(q)$ with reference linearization $J_h=D\mathcal A_h(q_0)$, the physical blind space after quotienting out the rigid-motion space $\Rr$ and the gauge $\Gg$ is
\begin{equation}\label{eq:physical-kernel}
 \ker J_h/(\Rr+\Gg).
\end{equation}
The complete recurrence, Fourier symbols, and root-multiplicity derivations are given in Appendix~\ref{app:kernels}; the discussion below uses the resulting kernel structure and its physical classification.

\begin{figure}[t]
 \centering
 \includegraphics[width=\textwidth]{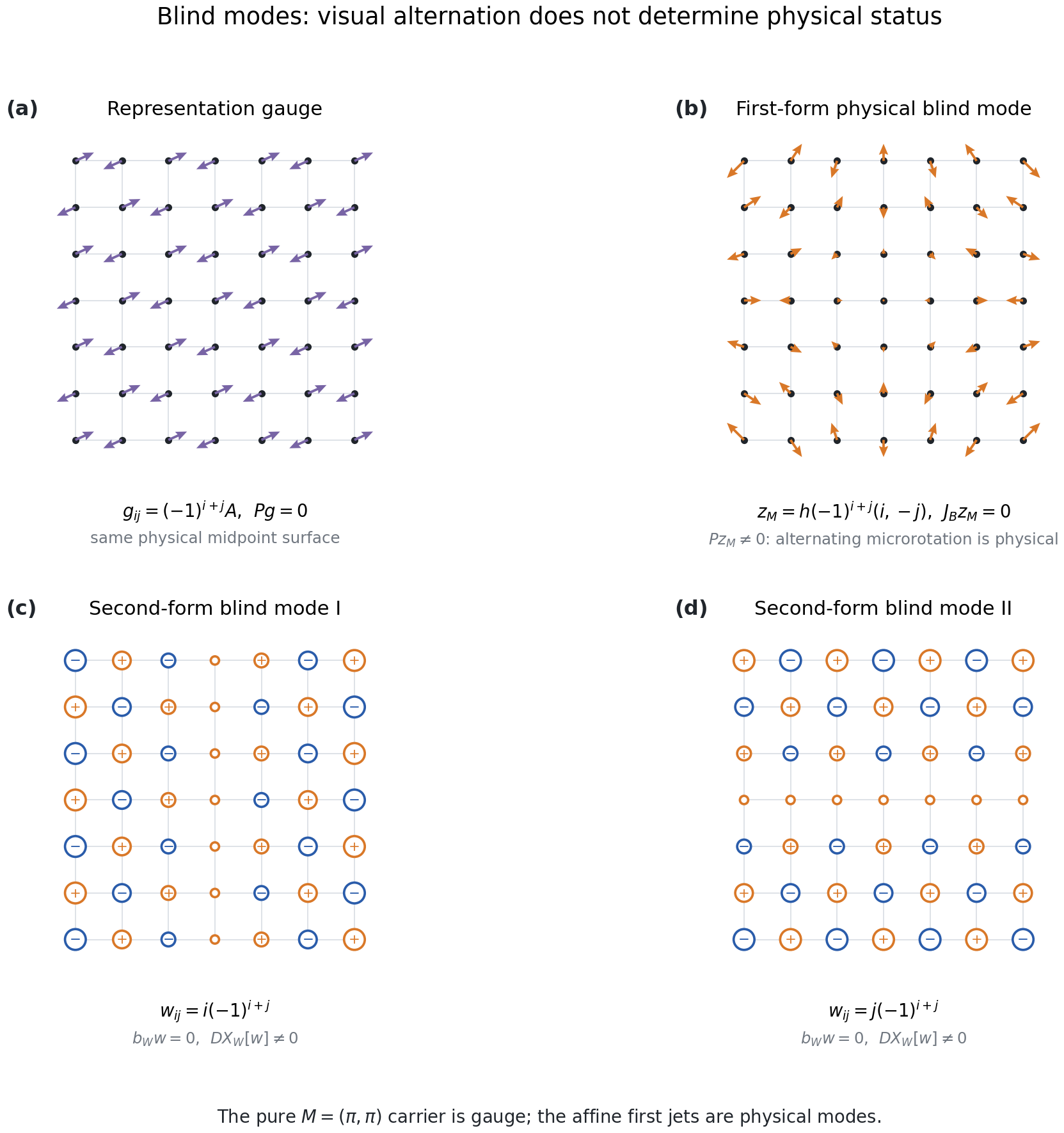}
 \caption{Representation gauge versus physical blind modes. The pure checkerboard carrier $p=(-1)^{i+j}$ leaves all edge midpoints fixed and is therefore a raw representation gauge. The affine-envelope mode $z_M$ produces nonzero physical midpoint motion while remaining invisible to the linearized B metric. The transverse modes $ip$ and $jp$ belong to the physical curvature-blind space of the symmetric $b_W$. The classification is determined by the physical midpoint response and the kernel of the observation operator, not by the visual checkerboard pattern alone.}
 \label{fig:blind}
\end{figure}
\FloatBarrier

\subsection{One-Dimensional Quotient Defect of the First Fundamental Form}\label{sec:first-blind}

The first-form kernel analysis asks whether the B metric misses a genuine in-plane deformation. Because $a_B$ is the Gram matrix of the B frame, it naturally records lengths and angles but does not directly record an infinitesimal rotation of that frame. In the continuum, spatially varying local rotations are restricted by displacement compatibility and cannot generally occur independently. The relevant question is whether the shared-node Checkerboard lattice admits an additional alternating rotation field that preserves the local Gram data while still moving the physical midpoints.

At a flat reference state, write the tangent variation as $\delta T_B=T_B^0F_B$. The differential of the Gram map is
\begin{equation}\label{eq:gram-diff}
 D(T_B^TT_B)[T_B^0F_B]=F_B^T+F_B=2\sym F_B.
\end{equation}
Local stretch, compression, and shear are therefore observed, whereas the skew part of $F_B$ is not. Enforcing compatibility of this local freedom through the shared raw nodes restricts the cell microrotation to
\begin{equation}
 \omega_{ij}=c_0+c_M(-1)^{i+j}.
\end{equation}
The constant term $c_0$ is an ordinary rigid yaw; the alternating term integrates to
\begin{equation}\label{eq:zm}
 z_M(i,j)=h(-1)^{i+j}(i,-j).
\end{equation}

\begin{theorem}[B-metric quotient kernel]\label{thm:first-kernel}
For a connected regular flat structured patch, the in-plane linearization of the B metric satisfies
\begin{equation}\label{eq:first-kernel}
 \ker J_B^{\rm in}=\Rr^{\rm in}\oplus\Gg^{\rm in}
 \oplus\operatorname{span}\{z_M\}.
\end{equation}
In particular, the physical quotient defect is one-dimensional.
\end{theorem}

The mechanical content of \cref{thm:first-kernel} is the separation of representation redundancy from a genuine defect of the membrane observation. The pure $M=(\pi,\pi)$ carrier in Fig.~\ref{fig:blind}(a) leaves all edge midpoints fixed and therefore only changes the raw representative of the same physical surface. By contrast, the mode $z_M$ in Fig.~\ref{fig:blind}(b) moves the physical midpoint surface while still satisfying $Da_B[z_M]=0$. A membrane energy built only from $a_B$ would therefore assign an additional zero-stiffness in-plane mechanism to $z_M$. The issue is not a global shortage of stiffness, but a precisely identifiable one-dimensional loss of observability in the physical quotient space.

In Fourier language, $z_M$ is the affine first jet of the pure $M$ carrier. The shared-node recurrence, direct-sum independence, and finite-patch dimension count are given in Appendix~\ref{app:kernels}. This one-dimensional defect in turn dictates that the membrane completion should be as restrained as possible: it should remove $z_M$ and no more.

\subsection{Two Physical Blind Modes of the Second Fundamental Form}\label{sec:second-blind}

The second-form kernel analysis asks whether $b_W$ misses genuine bending deformation. For a flat reference configuration, the first-order change of the normal can be represented by the transverse scalar displacement $w_{ij}$. Because $b_W$ is constructed from diagonal differences of B-face normals, its linearization acts on a diagonal graph that preserves checkerboard parity. It is therefore necessary to determine whether the even and odd sublattices retain fields that are invisible to this discrete-Hessian-type observation.

Let $p_{ij}=(-1)^{i+j}$. Solving the lattice compatibility relations of $Db_W$ gives the following result.
\begin{theorem}[Flat curvature kernel]\label{thm:curvature-kernel}
On a regular rectangular patch,
\begin{equation}\label{eq:curvature-kernel}
 \ker J_{b_W}=\operatorname{span}\{1,i,j,p,ip,jp\}.
\end{equation}
Here $1,i,j$ are transverse rigid motions, $p$ is the scalar checkerboard gauge, and $ip,jp$ are two non-rigid, non-gauge physical blind modes.
\end{theorem}

The six directions are distinguished by whether they move the physical midpoint surface. The constant mode $1$ is a transverse translation, and $i$ and $j$ are infinitesimal rigid rotations about the two in-plane directions; these belong to the rigid kernel of the continuous shell. The pure carrier $p$ takes opposite values at the two ends of every raw edge, so all edge-midpoint averages vanish and the mode is the scalar counterpart of the representation gauge in Section~\ref{sec:quotient}. The two remaining first jets are different:
\[
 ip=i(-1)^{i+j},\qquad jp=j(-1)^{i+j}.
\]
In $ip$, the amplitude of the alternating pattern varies linearly in the $i$ direction; in $jp$, it varies linearly in the $j$ direction. Both satisfy $P(ip)\neq0$ and $P(jp)\neq0$, so both genuinely move the physical midpoint surface, yet they also satisfy $Db_W=0$. Thus the symmetric W-curvature observation produces no first-order curvature response to either mode.

Figures~\ref{fig:blind}(c) and \ref{fig:blind}(d) show these two modes separately. Their checkerboard appearance resembles the representation gauge in Fig.~\ref{fig:blind}(a), but their edge-midpoint displacements are nonzero. The missing information therefore belongs to the W-centered symmetric curvature observation itself. Mathematically, the two modes are the affine first jets of the $M$-point carrier $p$ on the parity-preserving diagonal stencil. The full parity-affine recurrence and Fourier root-multiplicity proof are given in Appendix~\ref{app:kernels}.

A bending energy built only from $b_W$ would therefore retain two additional zero-stiffness curvature mechanisms. Together with the one-dimensional membrane defect identified above, the missing information is now completely characterized: one membrane direction and two curvature directions. Their completion must remove these underconstrained modes without violating the distinct thickness scalings of membrane and bending energies.
\section{Scale-Compatible Completion and Finite-Configuration Shell Energy}\label{sec:completion}

\subsection{Design Constraints for Minimal Membrane and Curvature-Compatible Completion}\label{sec:completion-principle}

Section~\ref{sec:blind} identified exactly which physical information is missing: one membrane direction and two curvature directions. The completion must remove these underconstraints while preserving, as far as possible, the near-isometric deformation space of the continuous thin shell. This requirement is especially stringent on the membrane side. Any additional $O(t)$ constraint beyond what is needed to remove the defect assigns membrane stiffness to deformation that should be governed primarily by bending, and the resulting artificial stiffness is strongly amplified in the thin-shell limit. The distinct thickness scalings of membrane and bending therefore impose different design requirements on the two completion terms.

Near the reference state, the discrete tangent can be written as
\begin{equation}\label{eq:completion-scale}
 K(t)=tK_m+t^3K_b.
\end{equation}
If an unnecessary stiffness $\delta K_m$ is added to the membrane sector, its magnitude relative to the physical bending stiffness scales as
\begin{equation}\label{eq:membrane-amplification}
 \frac{t\,\delta K_m}{t^3K_b}\sim\frac{\delta K_m}{t^2}.
\end{equation}
Hence even a small $\delta K_m$ can suppress bending as $t\to0$ through the relative $t^{-2}$ amplification and thereby induce membrane locking. Combined with the one-dimensional quotient defect of \cref{thm:first-kernel}, this makes the membrane completion highly constrained: it should provide restoring stiffness only in the missing $z_M$ direction, preserve rigid motions and the raw gauge, vanish on homogeneous affine membrane deformation, and avoid unnecessarily shrinking the near-isometric space of the shell. At the linear-observation level, the minimum sufficient added rank for a one-dimensional defect is one.

The curvature sector is different. An $X_W$-type completion enters the same $t^3$ sector as the physical bending stiffness, in the form $t^3(K_b+K_{X_W})$, and therefore does not acquire a $t^{-2}$ amplification relative to bending. It need not satisfy the same minimum-rank requirement as the membrane completion. Instead, it must respond to the two blind modes $ip$ and $jp$, remain objective under finite rotations, preserve smooth consistency, and avoid dominating the physical low-frequency bending response at finite resolution.

Accordingly, the membrane side uses a quotient-minimal completion that restricts the added $O(t)$ constraint to the single missing direction, whereas the curvature side uses an objective compatibility completion that restores the missing normal information within the existing $O(t^3)$ bending sector. These principles lead to $X_M$ and $X_W^{\mathrm{rel}}$, respectively.

\subsection{Intrinsic Membrane Compatibility Coordinate \texorpdfstring{$X_M$}{XM}}\label{sec:xm}

The first fundamental form already records the Gram data of the B frame. What it misses is the compatibility of alternating microrotations, which is invisible to the Gram map. The membrane completion should therefore not introduce another measure of stretch or shear; it should detect only whether neighboring B-frame rotations contain the $M$-point variation associated with $z_M$. A suitable supplementary observation should vanish on homogeneous affine membrane deformation, be nonzero on $z_M$, and remain invariant under rigid motions and the raw gauge.

The construction of $X_M$ follows four requirements: invariance under superposed rigid motions, invariance under the raw gauge, exact vanishing on homogeneous affine membrane deformation, and nonzero response to $z_M$. Since the physical quotient defect is one-dimensional, only one scalar observation is required. The implementation first uses the polar factor to separate stretch from rotation in each B frame, compares the reference-relative rotations of neighboring B frames, measures their mixed connection variation over a plaquette, and finally extracts the global $M$ character as a scalar moment.

For a nondegenerate, orientation-preserving B cell, define the polar-normalized frame
\begin{equation}\label{eq:polar-frame}
 E_B=T_Ba_B^{-1/2},\qquad
 \mathcal E_B=[E_{B,1},E_{B,2},E_{B,1}\times E_{B,2}]\in SO(3).
\end{equation}
The cell rotation relative to the reference configuration is
\begin{equation}
 R_B(q;q_0)=\mathcal E_B(q)\mathcal E_B(q_0)^T.
\end{equation}
For adjacent cells, define the objective relative connection $C_B^\alpha=R_B^TR_{B+e_\alpha}$. On the admissible set where the rotation angle does not cross the $\pi$ cut locus, use the Cayley axial map
\begin{equation}
 \vartheta(C)=2\operatorname{axl}\big[(C-I)(C+I)^{-1}\big].
\end{equation}
Let
\begin{equation}\label{eq:connection-reference-normal}
 \bar N_B^\alpha=
 \frac{N_B^0+N_{B+e_\alpha}^0}
 {\norm{N_B^0+N_{B+e_\alpha}^0}},
 \qquad
 \theta_B^\alpha=\bar N_B^\alpha\cdot\vartheta(C_B^\alpha),
\end{equation}
where $\bar N_B^\alpha$ is the normalized average reference normal associated with the $\alpha$-connection. On a complete B-cell plaquette, define the mixed moment
\begin{equation}\label{eq:mixed-moment}
 \mu_{ij}=\frac12\left[(\theta^x_{i,j+1}-\theta^x_{ij})
 +(\theta^y_{i+1,j}-\theta^y_{ij})\right].
\end{equation}
Let $A_{ij}^{M,0}$ denote the mean of the four neighboring B-reference areas covered by the $M$-plaquette indexed by $(i,j)$, and let $A_M^0=\sum A_{ij}^{M,0}$. The symbol $M$ is used here to denote the alternating plaquette and to avoid confusion with an ordinary cell area. The membrane compatibility coordinate is
\begin{equation}\label{eq:xm}
 \XM(q;q_0)=\frac{1}{4\sqrt{A_M^0}}
 \sum_{ij}A_{ij}^{M,0}(-1)^{i+j}\mu_{ij}(q;q_0).
\end{equation}
$X_M$ is a topology-fixed $M$-character moment. It does not depend on a benchmark kernel, load, or thickness. It is a single global scalar and therefore produces a rank-one update in the linear tangent. The coefficient paired with this normalization is
\begin{equation}\label{eq:beta-xm}
 \beta_M=8\mu=\frac{4E}{1+\nu}.
\end{equation}
This coefficient normalizes the missing $M$-point shear channel to the same trace-free shear modulus as the continuous isotropic material; no benchmark fitting is involved. The matching calculation is given in Appendix~\ref{app:completion-coefficients}.

\begin{theorem}[Quotient-minimal membrane completion]\label{thm:xm-minimal}
$X_M$ is exactly invariant under superposed rigid motions and the raw checkerboard gauge, vanishes for flat affine membrane strains, and satisfies
\begin{equation}\label{eq:xm-joint-kernel}
 \ker(J_B,D\XM)=\Rr^{\rm in}\oplus\Gg^{\rm in}.
\end{equation}
Because the quotient defect in \cref{thm:first-kernel} is one-dimensional, any linear supplement that removes it must have rank at least one. Thus $X_M$ is minimal in this quotient-rank sense.
\end{theorem}

Here ``minimum rank'' refers specifically to the minimum linear observation required to remove a one-dimensional quotient defect. It does not impose any broader statement beyond the flat linearized setting. The algebra establishing objectivity, affine exactness, and $DX_M[z_M]\neq0$ is given in Appendix~\ref{app:objectivity}.

\subsection{Objective Curvature Compatibility Coordinate \texorpdfstring{$X_W^{\mathrm{rel}}$}{XW-rel}}\label{sec:xw}

The curvature completion targets a different type of missing information. Section~\ref{sec:second-blind} showed that $ip$ and $jp$ move the physical midpoint surface without changing the linearized symmetric $b_W$. In a flat reference configuration, the alternating sum of the four neighboring B normals responds to both modes and therefore provides a direct probe of the missing curvature compatibility.

A curved reference configuration and finite rotation introduce an additional requirement: current and reference compatibility vectors must be compared in consistent frames. If two spatial vectors are subtracted directly in a fixed global frame, a pure rigid rotation generates a spurious difference. The reference compatibility vector must therefore be transported with the local surface rotation before it is compared with the current quantity. We construct this transport from the surface-polar factor of the W-centered tangent deformation map. A single W rotation is applied to the reference quantity, which preserves both rigid-motion objectivity and the alternating normal component that must remain observable.

Define the alternating sum of the four neighboring B normals by
\begin{equation}\label{eq:xw-raw}
 X_W=N_{NE}-N_{NW}+N_{SW}-N_{SE}.
\end{equation}
In the flat linearization, $X_W$ is nonzero on $ip$ and $jp$, and the joint kernel retains only $1,i,j,p$. For a curved reference, the direct difference $X_W(q)-X_W(q_0)$ is not objective. Define the W-centered unit normal and three-dimensional frame by
\begin{equation}\label{eq:w-frame}
 \begin{aligned}
 n_W(q)&=\frac{t_1^W(q)\times t_2^W(q)}
 {\norm{t_1^W(q)\times t_2^W(q)}},\\
 E_W(q)&=[\,t_1^W(q),\;t_2^W(q),\;n_W(q)\,],
 \qquad E_W^0=E_W(q_0).
 \end{aligned}
\end{equation}
On the admissible set $F_W\in GL^+(3)$, the current/reference deformation map and its orientation-preserving right polar rotation are
\begin{equation}\label{eq:w-polar}
 \begin{aligned}
 F_W(q,q_0)&=E_W(q)[E_W(q_0)]^{-1},\\
 R_W(q,q_0)&=F_W(q,q_0)
 [F_W(q,q_0)^T F_W(q,q_0)]^{-1/2}\in SO(3).
 \end{aligned}
\end{equation}
The reference-relative curvature compatibility coordinate is then
\begin{equation}\label{eq:xwrel}
 \XW(q;q_0)=X_W(q)-R_W(q,q_0)X_W(q_0).
\end{equation}
Under a superposed rigid rotation $Q$, $X_W(q)\mapsto QX_W(q)$ and $R_W\mapsto QR_W$, so $X_W^{\rm rel}\mapsto QX_W^{\rm rel}$ and its Euclidean norm is objective. For a smooth surface family, $X_W^{\rm rel}=O(h^2)$.

The curvature completion also uses an explicit, non-fitted coefficient. Let $g_W^0:=a_W(q_0)$ be the local reference W metric and $h_1,h_2$ the mesh spacings in the two parametric directions. Exact integration of an auxiliary-Q1 normal-gradient energy yields the geometric factor
\begin{equation}\label{eq:gamma-xw}
 \gamma_W^0=\frac{\sqrt{\det g_W^0}}{12}
 \left[(g_W^0)^{11}\frac{h_2}{h_1}+(g_W^0)^{22}\frac{h_1}{h_2}\right],
\end{equation}
and therefore the coefficient multiplying $X_W^{\mathrm{rel}}$ in the bending energy is
\begin{equation}\label{eq:beta-xw}
 \beta_W^0=\mu\gamma_W^0
 =\frac{\mu\sqrt{\det g_W^0}}{12}
 \left[(g_W^0)^{11}\frac{h_2}{h_1}+(g_W^0)^{22}\frac{h_1}{h_2}\right].
\end{equation}
For an orthogonal square grid, $\gamma_W^0=1/6$ and $\beta_W^0=\mu/6$. The coefficients $\beta_M$ and $\beta_W^0$ are fixed, respectively, by the constitutive normalization of the missing membrane channel and by the normal-gradient energy of the W patch; neither depends on benchmark tuning. Objectivity of the surface-polar transport is detailed in Appendix~\ref{app:objectivity}, and both coefficient derivations are given in Appendix~\ref{app:completion-coefficients}.

\subsection{Finite-Configuration Discrete Potential Energy}\label{sec:energy}

With the membrane and curvature compatibility observations in place, both must be incorporated into one finite-configuration energy at the correct thin-shell scales. The construction preserves the membrane--bending separation of the Koiter model: $a_B$ and $X_M$ together represent $O(t)$ membrane information, whereas $b_W$ and $X_W^{\rm rel}$ represent $O(t^3)$ bending and normal-compatibility information. The completion terms thereby repair the observability of the discrete geometry without changing the physical membrane and bending scalings of the continuum model.

For a positive-definite reference metric $A$ and a symmetric tensor $Z$, use the plane-stress St.~Venant--Kirchhoff/Koiter contraction
\begin{equation}\label{eq:qa}
 Q_A(Z)=\frac{\alpha}{2}[\tr(A^{-1}Z)]^2
 +\mu\tr[(A^{-1}Z)^2],
 \quad
 \alpha=\frac{E\nu}{1-\nu^2},\quad
 \mu=\frac{E}{2(1+\nu)}.
\end{equation}
The membrane energy is
\begin{equation}\label{eq:membrane-energy}
 E_m(q;q_0)=\frac t4\left[
 \sum_B A_B^0Q_{a_B^0}(a_B(q)-a_B(q_0))+\beta_M\XM(q;q_0)^2
 \right],\qquad \beta_M=8\mu.
\end{equation}
Here $A_B^0$ is the reference physical quadrature area. The value of $\beta_M$ is paired with the $M$-character normalization in \cref{eq:xm}, so that the trace-free shear channel missed by $a_B$ carries the same shear modulus as in the continuous isotropic membrane energy. It is independent of load, thickness, or benchmark tuning.

The bending and curvature-compatibility energies are
\begin{align}
 E_b(q;q_0)&=\frac{t^3}{12}\sum_{W\,\mathrm{all}}
 \theta_WA_W^0Q_{a_W^0}(b_W(q)-b_W(q_0)),\label{eq:bending-energy}\\
 E_{X_W}(q;q_0)&=\frac{t^3}{12}
 \sum_{W\in\mathcal W_{\rm int}^{\rm complete}}
 \beta_W^0\norm{\XW(q;q_0)}^2,\qquad
 \beta_W^0=\mu\gamma_W^0.\label{eq:xw-energy}
\end{align}
Here $A_W^0$ denotes the full reference dual area associated with a W site, and $\theta_W\in\{1,1/2,1/4\}$ is the interior/edge/corner dual fraction. The curvature-compatibility coefficient $\beta_W^0$ is given by \cref{eq:beta-xw}: $\mu$ provides the material bending scale, while $\gamma_W^0$ encodes the metric and aspect ratio of the reference W patch. The $E_{X_W}$ term is accumulated only on complete interior W patches, consistent with the complete patch used in the auxiliary-Q1 derivation. With external-load potential $\mathcal W_{\rm ext}$, the total potential is
\begin{equation}\label{eq:potential}
 \Pi_h(q;\lambda)=E_m+E_b+E_{X_W}-\mathcal W_{\rm ext}(q;\lambda),
 \qquad R= D_q\Pi_h=0.
\end{equation}
The nonlinear residual and consistent tangent $K_T=D_q^2\Pi_h$ are both obtained by automatic differentiation of the same finite-configuration energy.

The resulting potential still uses nodal positions as its only independent unknowns, while the internal geometry is evaluated exclusively on the physical midpoint/B-face surface. A finite-configuration thin-shell energy should additionally satisfy three basic requirements: changing the raw representative or superposing a rigid motion must not create artificial energy; the minimal membrane completion should preserve the principal near-isometric branches of curved shells; and the influence of residual membrane stiffness in the fixed-$h$ ultra-thin regime should be explicitly interpretable. Section~\ref{sec:analysis} addresses these properties in turn.
\section{Analytical Properties}\label{sec:analysis}

\subsection{Representation Invariance, Rigid-Motion Objectivity, and Reference-State Consistency}

A finite-configuration shell energy must first be independent of the chosen representation and objective under Euclidean rigid motions: internal energy should respond only to actual material deformation. In the present representation, this requirement has three distinct components. First, the same midpoint surface may admit different raw representatives, and changing the representative must not change the energy. Second, an arbitrary rigid translation or rotation superposed on the entire shell must leave all material strains and the internal energy unchanged. Third, the reference configuration itself must be stress free. These conditions correspond to gauge invariance, rigid-motion objectivity, and reference-state consistency.

\begin{proposition}\label{prop:invariance}
On the admissible set where all B/W frames are nondegenerate and the polar/Cayley maps are well defined, the internal energy satisfies:
\begin{enumerate}
 \item for every raw gauge $g_{ij}=(-1)^{i+j}A$, $E_{\rm int}(q+g;q_0)=E_{\rm int}(q;q_0)$;
 \item for every $Q\in SO(3)$ and $c\in\R^3$, a superposed rigid motion of the current configuration leaves the internal energy unchanged;
 \item $E_{\rm int}(q_0;q_0)=0$ and $DE_{\rm int}(q_0;q_0)=0$.
\end{enumerate}
\end{proposition}

The first property follows from the exact gauge invariance of all B diagonal differences. For the second, the metric and $b_W$ are invariant, $X_W^{\rm rel}$ is covariant, and $X_M$ is invariant because it is built from relative rotations. The third follows because all reference-relative residual quantities vanish at $q=q_0$. These are the basic objectivity requirements of geometrically exact shell theory \cite{simoFox1989,sauerDuong2017}. Gauge invariance states that different raw representatives describe the same physical state, whereas rigid objectivity expresses Euclidean frame indifference; the two cannot be replaced by the same type of boundary fixing. Componentwise proofs are provided in Appendix~\ref{app:objectivity}.

Proposition~\ref{prop:consistency} further ensures that $a_B$ and $b_W$ approximate the smooth continuous fundamental forms to second order, that $X_W^{\rm rel}=O(h^2)$, and that $X_M$ vanishes exactly for homogeneous affine membrane deformation. These properties exclude artificial material response under rigid motion, in the reference state, and for uniform affine membrane strain. The remaining question is whether the membrane completion preserves the near-isometric bending motions permitted by the continuum theory on a curved surface.

\subsection{Near-Isometry Approximation on Aligned Generalized Cylinders}\label{sec:cylinder-theorem}

The flat kernel establishes that $X_M$ removes the target checkerboard mechanism, but it does not show whether a sufficiently rich near-isometric space survives on a curved surface. The dominant bending response of a thin shell follows near-isometric motion of the continuous midsurface. If the completed discrete membrane kernel could not approximate such motion, a perfectly correct flat nullity count would still be compatible with artificial membrane stiffness on curved shells.

A generalized cylinder provides a useful analytically tractable test class. It has nonzero initial curvature and is therefore nontrivial from the viewpoint of shell bending, while at the same time it admits an explicit family of infinitesimal isometries. This makes it possible to estimate the distance between a continuous isometric field and the completed discrete membrane kernel. Under arc-length sampling and alignment of the grid with the generator direction, this distance can be bounded directly.

\begin{theorem}[Near-isometry approximation on an aligned generalized cylinder]\label{thm:cylinder}
Let the reference surface be a generalized cylinder
\[
 r(s,z)=c(s)+ze_z,
 \qquad c'(s)\cdot e_z=0,\qquad |c'(s)|=1,
\]
with $c\in C^4$ and bounded signed curvature. Define $e_\theta=c'(s)$ and $n=e_\theta\times e_z$, with the sign convention $e_\theta'=-\kappa n$. Consider a structured mesh aligned with the $s$ and $z$ directions, belonging to a shape-regular family and having uniformly nondegenerate section chords. Let
\[
 v(s,z)=u(s)e_\theta(s)+w(s)n(s)+c_z e_z,
 \qquad u'(s)+\kappa(s)w(s)=0,
\]
with $u,w\in C^3$ and constant $c_z$, so that $v$ is a row-independent infinitesimal isometry (up to the rigid translation $c_ze_z$). For a free patch, or for a patch whose essential trace is compatible with the discrete representative constructed below, there exists, for sufficiently small $h$, a discrete field $v_h^C$ in the exact completed membrane kernel such that
\begin{equation}\label{eq:cylinder-error}
 \norm{I_hv-v_h^C}_{0,h}+h\abs{I_hv-v_h^C}_{1,h}\le Ch^2,
\end{equation}
where $I_hv$ denotes nodal sampling of $v$, and $\norm{\cdot}_{0,h}$ and $\abs{\cdot}_{1,h}$ are the standard trapezoidal discrete $L^2$ norm and first-difference $H^1$ seminorm on the structured grid. The constant $C$ is independent of $h$ but depends on the domain length, the indicated regularity bounds of $c$ and $v$, chord nondegeneracy, and mesh shape regularity.
\end{theorem}

The shared-node B-metric recurrence first constructs a row-independent exact-kernel branch. Because the B-cell rotations in this branch are identical along the generator rows, the axial relative connection is the identity and the circumferential connection is independent of the row index. Both mixed connection differences in $X_M$ therefore vanish pointwise, without any additional parity assumption. An arc-length chord expansion then gives the second-order discrete $L^2$ and scaled $H^1$ errors; the full construction is given in Appendix~\ref{app:cylinder}. The theorem establishes only an $O(h^2)$ approximation property of the completed membrane kernel for the stated aligned, row-independent family of infinitesimal isometries. It is not a general solution-convergence theorem for arbitrary curved shells and does not imply uniform locking-free convergence.

The result extends the membrane-side design principle of Section~\ref{sec:completion-principle} from the flat kernel to a nontrivial curved-surface class. Removing $z_M$ does not eliminate the natural infinitesimal-isometry branch of the generalized-cylinder family, and the exact-kernel field approaches the continuous target at order $O(h^2)$. By contrast, treating a reconstructed W metric as a second full-strength membrane strain would impose additional membrane constraints on the same branch. The theorem therefore provides analytical evidence for curved-surface near-isometry representability, rather than a claim of locking-free behavior on arbitrary surfaces.

\subsection{Thin-Shell Modal Scaling and Fixed-\texorpdfstring{$h$}{h} Residual Locking}\label{sec:modal}

Removing exact blind modes does not automatically eliminate the ultra-thin fixed-$h$ locking limit. Load-bearing bending modes generally belong to a membrane \emph{near}-kernel rather than to the exact membrane kernel, and can therefore have small but nonzero membrane generalized eigenvalues at finite resolution. As thickness decreases, the $O(t)$ membrane contribution is progressively amplified relative to the $O(t^3)$ bending contribution and can suppress the physical bending response. The reduced tangent at the reference equilibrium retains the membrane--bending decomposition of \cref{eq:completion-scale}.

On the rigid/gauge-reduced load-bearing subspace, and for a mode on which the bending quadratic form is nonzero, consider the generalized eigenproblem
\begin{equation}
 K_m\phi_j=\lambda_jK_b\phi_j.
\end{equation}
Let the applied load be $t^3f_b$, and normalize the response by the pure-bending solution. Projection onto $\phi_j$ gives the denominator $t\lambda_j+t^3$, so the normalized modal response contains the factor
\begin{equation}\label{eq:attenuation}
 \frac{t^3}{t\lambda_j+t^3}=\frac{t^2}{\lambda_j+t^2}.
\end{equation}
At fixed $h$, any relevant $\lambda_j>0$ causes the response to attenuate as $t\to0$. If refinement drives the load-bearing near-kernel eigenvalues toward zero, the onset of locking is shifted to smaller thicknesses. The thickness--refinement map in \cref{fig:m3} is explained by this finite-dimensional relation, which also clarifies how refinement pushes the residual membrane effect into a thinner regime. Exact physical blind modes correspond instead to underconstrained $\lambda_j=0$ directions; residual locking arises from a load-bearing family that should approach an isometry but retains $\lambda_j>0$ at fixed $h$.

These analytical properties establish the basic mechanical envelope of the finite-configuration energy: it is invariant under changes of representation and rigid motions, it preserves the target near-isometric branch on a nontrivial curved-surface class, and it makes explicit why residual membrane stiffness can still control the fixed-$h$ ultra-thin limit. The numerical tests below examine how these properties manifest in actual discrete shell responses.
\section{Numerical Results}\label{sec:numerics}

The numerical study proceeds from local geometric consistency to the spectral effect of the completion terms, classical linear-shell responses, and finally finite-configuration problems. Smooth-surface tests first examine convergence of the W-centered second form. Nullity, positive spectra, and standard load-bearing responses are then used to assess whether $X_M$ acts only on the target membrane mechanism while preserving the normal soft space. The plate, Scordelis--Lo roof, and M3 shell successively probe flat bending, membrane--bending coupling on an initially curved surface, and locking-sensitive near-isometric response. Three nonlinear benchmarks further cover large rotations, finite deformation of curved shells, and localized ovalization.

All Checkerboard results use \cref{eq:membrane-energy,eq:bending-energy,eq:xw-energy}. Linear benchmarks and MidedgeTan use identical material parameters and reference values. Nonlinear cases are solved from the finite-configuration energy with a consistent tangent, Newton line search, and adaptive load stepping. Unless noted otherwise, errors are measured against analytical solutions or verified literature references. Conventional quad-like surfaces shown in the figures are gauge-fixed representatives of the midpoint equivalence class, or are drawn directly from the physical B faces. Such reconstruction is used only for visualization; energy, internal forces, and tangents are always evaluated on the physical midpoint/B/W geometry.

\subsection{Geometric Convergence of the Second Fundamental Form}\label{sec:geometry-results}

The W-centered second form must first reproduce the continuous second fundamental form on smooth surfaces under mesh refinement. We test a cylinder, a sphere, a twisted surface, and a skew surface, using Dellinger's Checkerboard second form as an independent reference from the same geometric lineage \cite{dellinger2024}. Prior to comparison, diagonal directions, scales, and tensor-component conventions are aligned. Under this alignment the two first-form observations are locally algebraically identical, while the second forms are built, respectively, from a vertex-centered Gauss image and from Varignon B-face normals.

Figures~\ref{fig:geometry-convergence}(a)--(d) show the discrete $L^2$ errors on the four surfaces. The B-face/W-centered construction converges robustly at approximately second order in every case. The two curves are nearly coincident for the twist and skew tests, while the cylinder and sphere exhibit the same asymptotic order. Thus, directly differentiating the normals of strictly planar B faces preserves the smooth-consistency expected of Checkerboard curvature discretizations while retaining the unified geometric carrier required in Section~\ref{sec:second-form-geometry}. The subsequent shell formulation can therefore analyze gauge, blind modes, and finite-configuration objectivity without switching to a separate curvature representation.

\begin{figure}[t]
 \centering
 \includegraphics[width=\textwidth]{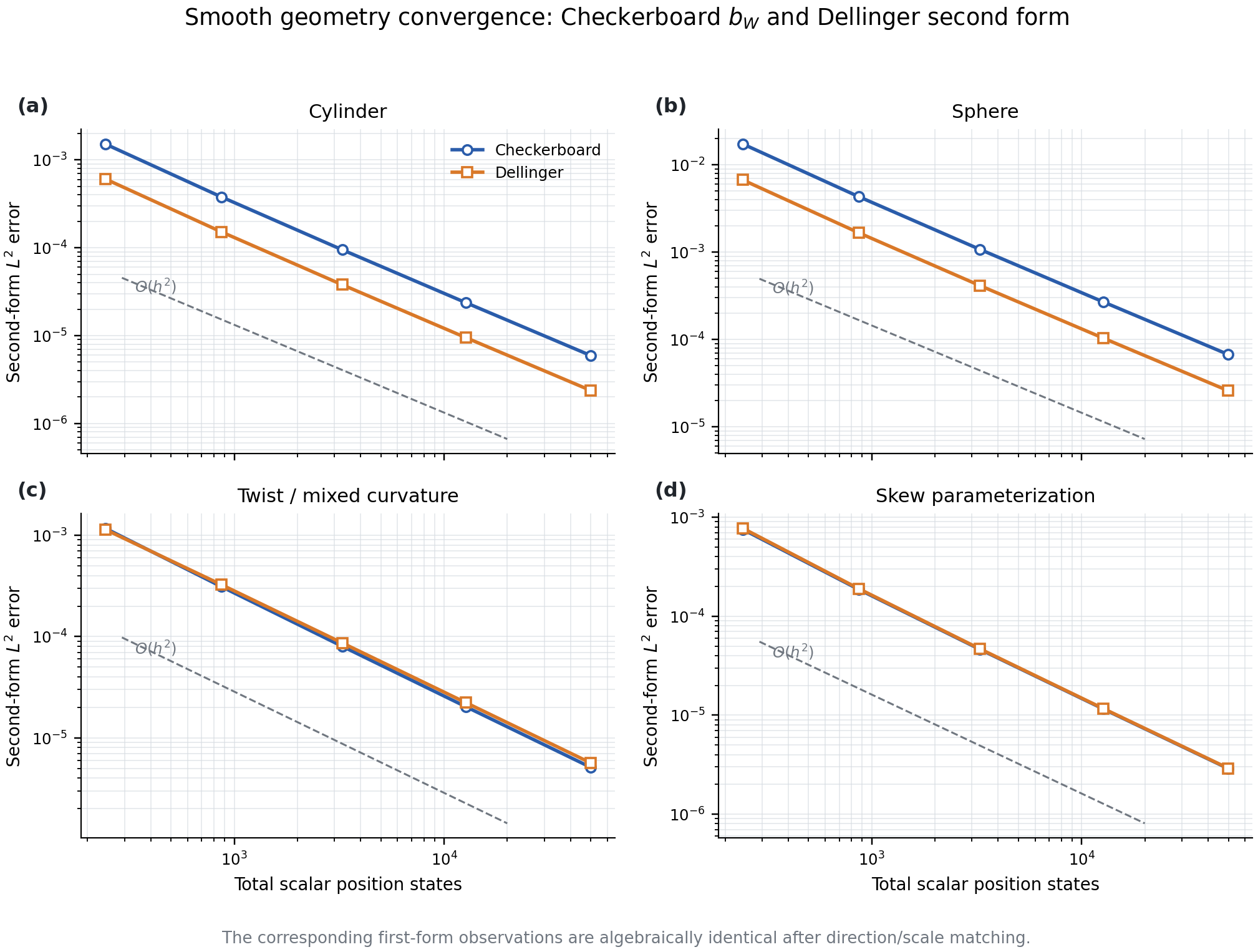}
 \caption{Discrete second-fundamental-form errors on four smooth surfaces. Panels (a)--(d) correspond to the cylinder, sphere, twist, and skew surfaces, respectively. The B-face/W-centered $b_W$ exhibits approximately second-order convergence in all four cases; an existing Checkerboard Gauss-image construction is included as an independent geometric reference. The horizontal axis is the total scalar DOF of the raw positional state.}
 \label{fig:geometry-convergence}
\end{figure}
\FloatBarrier

\subsection{Mechanism Verification of the Minimal Membrane Completion}\label{sec:mechanism-results}

Smooth geometric convergence does not reveal how strongly a completion term participates in the actual mechanical response. The minimal membrane completion should remove the unique physical membrane defect while leaving the original load-bearing soft space essentially unchanged. We therefore compare exact nullity, the active positive spectrum of M3, and standard load-bearing responses.

Figure~\ref{fig:mechanism}(a) shows the exact nullity of a flat patch. The B-only, $B+X_M$, and full B/W observations have nullities 6, 5, and 5, respectively, while the theoretical rigid-plus-gauge dimension is 5. Thus $X_M$ removes exactly one direction. Figure~\ref{fig:mechanism}(b) compares the M3 active positive spectrum. Apart from inserting stiffness in the removed direction, $B+X_M$ essentially preserves the smallest positive singular-value sequence of the B-side operator, whereas the full B/W observation shifts the same soft space upward by several orders of magnitude.

Figure~\ref{fig:mechanism}(c) shows the standard M3 load-bearing response. On the active spaces, the maximum relative displacement difference between $B$ and $B+X_M$ is $1.19\times10^{-10}$, and the maximum membrane-energy fraction carried by $X_M$ is $5.41\times10^{-12}$. On the fine meshes, the reference-relative $X_W^{\rm rel}$ changes the plate, Scordelis--Lo, and M3 displacements by 0.0136\%, 0.0210\%, and 0.2983\%, respectively. All three values are below $0.3\%$ and continue to decrease under refinement. Hence $X_W^{\rm rel}$ restores the missing curvature observability while introducing only a small, vanishing finite-mesh correction and does not become a dominant source of artificial bending stiffness. Together with rigid-motion objectivity and $O(h^2)$ smooth consistency, these results support the effectiveness and numerical non-intrusiveness of the reference-relative curvature completion in the linear shell problems considered here. Additional curved-patch nullity and sensitivity data are provided in Appendix~\ref{app:numerical}.

\begin{figure}[t]
 \centering
 \includegraphics[width=\textwidth]{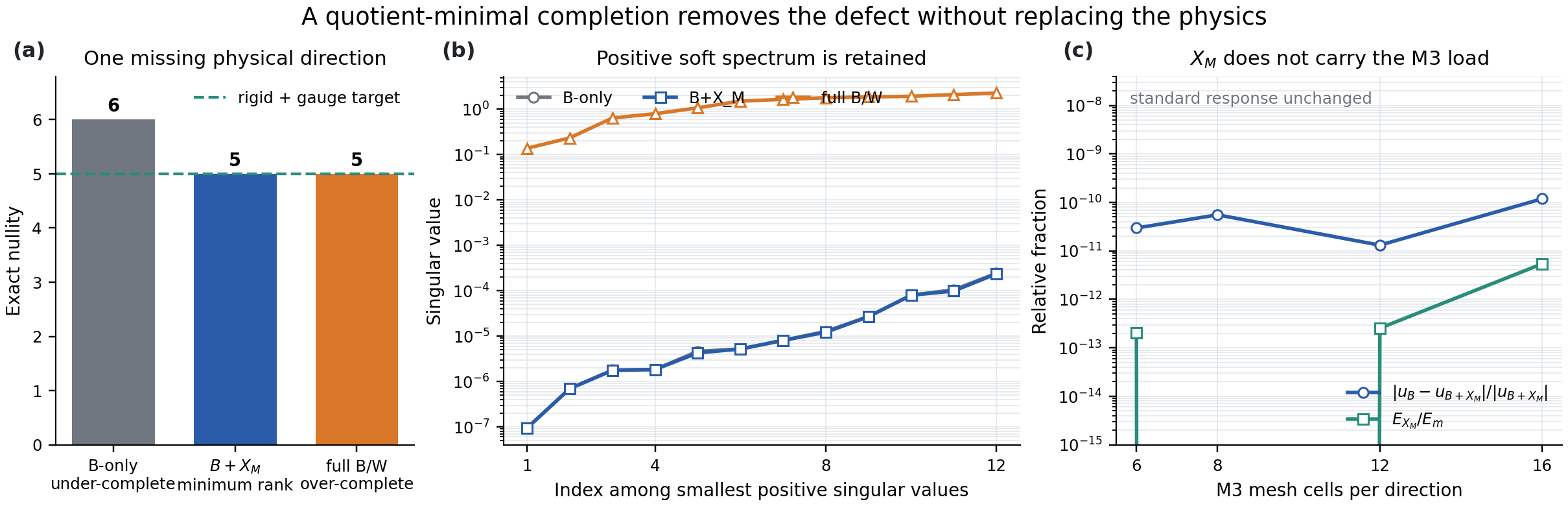}
 \caption{Spectral and load-response verification of the membrane compatibility completion. (a) Flat in-plane nullity: B-only has nullity 6, whereas adding $X_M$ reduces it to 5, exactly matching the rigid-plus-gauge dimension. (b) Smallest positive singular values on the $n=16$ M3 active space; $B+X_M$ preserves the low-end positive spectrum of the B-side operator. (c) M3 displacement difference between B-only and $B+X_M$, together with $E_{X_M}/E_m$.}
 \label{fig:mechanism}
\end{figure}
\FloatBarrier

\subsection{Linear Benchmarks and State Complexity}\label{sec:linear-benchmarks}

After the local action of the completion has been verified, standard load-bearing problems are used to assess the displacement response of the complete energy. The plate isolates pure bending on a flat surface and provides a basic flat-shell baseline without curvature-induced membrane locking. The Scordelis--Lo roof introduces initial curvature, distributed loading, rigid end diaphragms, and membrane--bending coupling within the same linear framework. Together, the two cases also compare the state complexity of the position-only formulation with MidedgeTan, which introduces independent edge states.

First consider a unit square simply supported plate with $E=1$, $\nu=0.3$, $t=0.01$, and uniform pressure $f=-10^{-8}e_z$ in the global $-z$ direction. Figure~\ref{fig:linear-benchmarks}(a) shows the simply supported condition on all four edges and the downward pressure. Figure~\ref{fig:linear-benchmarks}(b) overlays the medium-gray reference surface and the blue deformation amplified by a factor of 100, confirming the sign of the central displacement. The center deflection is reported as $w_c=-u_z(1/2,1/2)>0$, with Navier-series reference value $4.4360891\times10^{-4}$ \cite{timoshenkoWoinowskyKrieger1959}. In Fig.~\ref{fig:linear-benchmarks}(c), the Checkerboard errors for $n=8,16,32,64$ are 5.596\%, 1.488\%, 0.399\%, and 0.104\%, while the corresponding MidedgeTan errors are 4.953\%, 1.248\%, 0.313\%, and 0.078\%. Both methods converge steadily under refinement.

The Scordelis--Lo roof uses $R=25$, $L=50$, $t=0.25$, $E=4.32\times10^8$, $\nu=0$, and a uniform global $-z$ load of magnitude $q=90$. Figure~\ref{fig:linear-benchmarks}(d) shows the half-length domain, longitudinal symmetry, rigid end diaphragm, and downward loading; the standard observable is $-u_z$ at the free-edge midpoint. Figure~\ref{fig:linear-benchmarks}(e) overlays the medium-gray reference and the blue deformation amplified by a factor of 10 using the actual displacement $u=q-q_0$; the free edge moves in the global $-z$ direction. The classical reference used for the error plots is 0.3024 \cite{scordelisLo1964}, while modern computations give approximately 0.30059 \cite{sauerEtAl2024}. On the 12$\times$12 mesh, the free-edge displacement is $u_z=-0.33704$; local crown displacements may have the opposite sign without changing the sign of the standard free-edge observable. Figure~\ref{fig:linear-benchmarks}(f) gives Checkerboard errors of 11.45\%, 6.89\%, 3.66\%, and 2.54\% for $n=12,16,24,32$, respectively; the corresponding MidedgeTan errors are 9.10\%, 5.50\%, 2.20\%, and 0.856\%.

MidedgeTan introduces independent edge states, whereas Checkerboard retains only the raw positions. The discrete state counts are therefore approximately $3V+E$ and $3V$, respectively \cite{chenEtAl2026}. Figures~\ref{fig:linear-benchmarks}(c) and \ref{fig:linear-benchmarks}(f) plot error against total scalar state count, allowing accuracy and state size to be viewed on the same axis. Both formulations converge robustly, while Checkerboard reaches competitive plate and curved-shell displacement accuracy without edge or rotational variables. This state economy is a direct computational consequence of the unified position-only representation.

\begin{figure}[t]
 \centering
 \includegraphics[width=\textwidth]{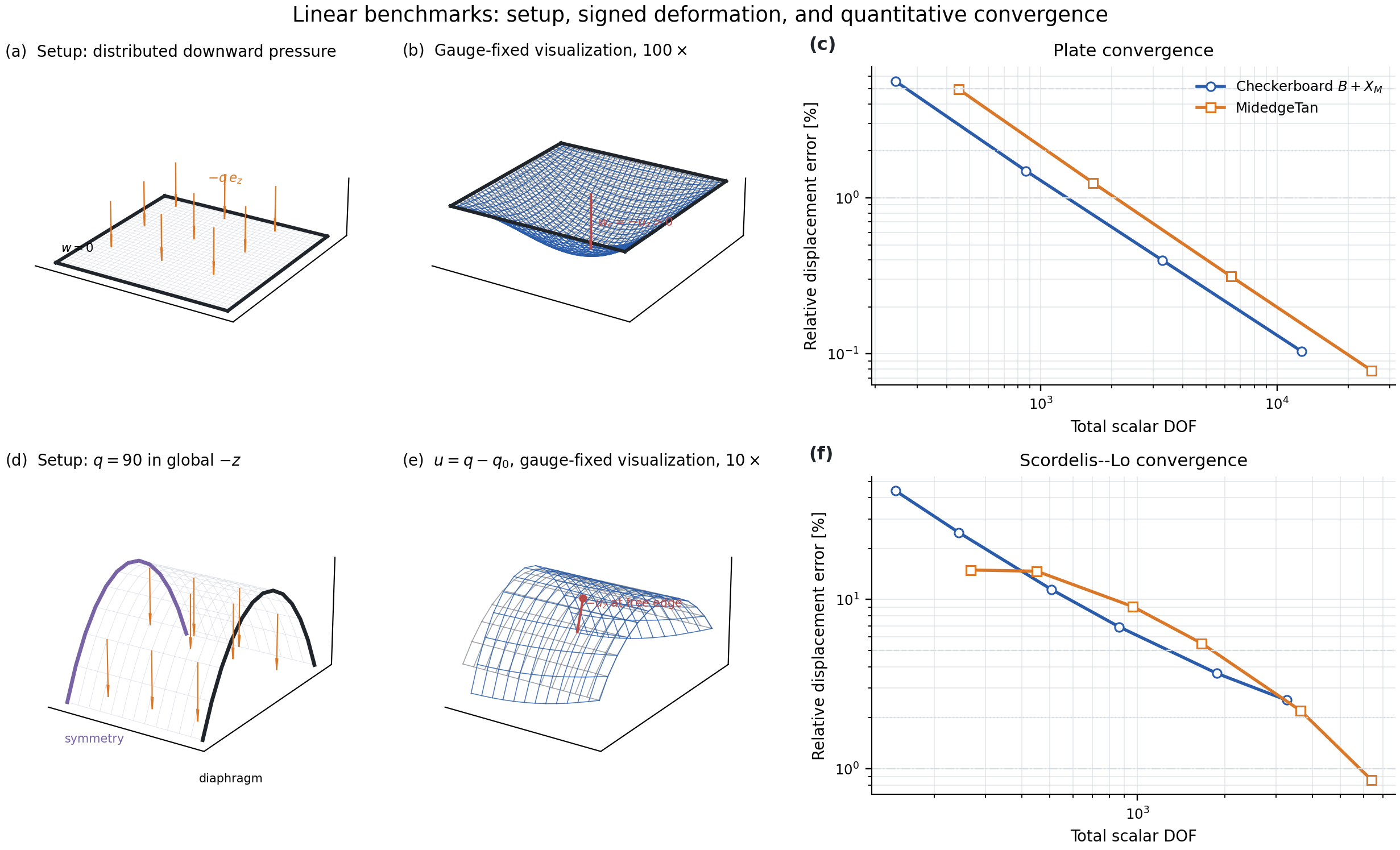}
 \caption{Setup, signed deformation, and quantitative convergence for the linear shell benchmarks. Top row: (a) simply supported plate with $w=0$ on the boundary and distributed global $-z$ pressure; (b) gray undeformed surface and gauge-fixed reconstructed downward deformation (100$\times$), with center observable $w_c=-u_z>0$; (c) error relative to the Navier solution versus total scalar DOF. Bottom row: (d) Scordelis--Lo half-length domain with longitudinal symmetry, a rigid end diaphragm, and multiple global $-z$ load arrows; (e) 10$\times$ deformation plotted from $u=q-q_0$ with free-edge observable $-u_z$; (f) error relative to 0.3024 versus total scalar DOF. Reconstruction is used only for visualization; the mechanics is evaluated on the physical midpoint/B-face geometry.}
 \label{fig:linear-benchmarks}
\end{figure}
\FloatBarrier

\subsection{M3 Benchmark and Thickness--Mesh Coupling}\label{sec:m3}

M3 is a classical locking-sensitive curved-shell benchmark. Its quarter hemisphere, cutout, and alternating forces generate a response dominated by near-isometric ovalization and provide a linearized reference-state, standard-thickness, and thickness-sensitive baseline. The same classical geometry is later used in the nonlinear hemisphere problem to extend the test from a linearized standard load to a finite-configuration equilibrium path.

The MacNeal--Harder M3 problem is a quarter hemisphere with an $18^\circ$ cutout, $R=10$, $E=6.825\times10^7$, $\nu=0.3$, and standard thickness $t=0.04$. In Fig.~\ref{fig:m3}(a), physical point $A=(R,0,0)$ is loaded outward in the global $+x$ direction and $B=(0,R,0)$ inward in the global $-y$ direction; the two radial edges satisfy reflection symmetry. The B-point observable is $-u_B=-\Delta y_B$. Classical benchmark studies place the reference response in the $0.093$--$0.094$ range \cite{macnealHarder1985,szeEtAl2004,neunteufelSchoberl2019}; the convergence plots and error values below use the fixed value 0.09352155 consistently throughout the benchmark suite. Figure~\ref{fig:m3}(b) shows only the medium-gray reference and the blue deformation amplified by a factor of 18; the blue B end moves into the reference sphere. Load directions and observables are confined to Fig.~\ref{fig:m3}(a). Figure~\ref{fig:m3}(c) gives the standard-thickness refinement: Checkerboard displacements for $n=6,8,12,16$ are 0.05403, 0.07870, 0.09146, and 0.09342, corresponding to errors of 42.23\%, 15.85\%, 2.20\%, and 0.1105\%. The refinement sequence rapidly approaches the standard reference.

Figure~\ref{fig:m3}(d) summarizes the thickness--refinement error map, while Fig.~\ref{fig:m3}(e) isolates the fixed-$h$ normalized-response tail. When the load is scaled with $t^3$, the normalized displacement at fixed $h$ still attenuates for sufficiently small $t$; refinement from $n=6$ to 16 shifts the onset to a thinner regime. This trend is consistent with the modal relation in \cref{eq:attenuation}: maintaining an accurate response deep in the thin-shell regime requires mesh and thickness to be refined together. $X_M$ removes the identified exact checkerboard membrane mechanism, while the accuracy of higher-order curved near-isometry approximation controls the residual ultra-thin response.

\begin{figure}[t]
 \centering
 \includegraphics[width=\textwidth]{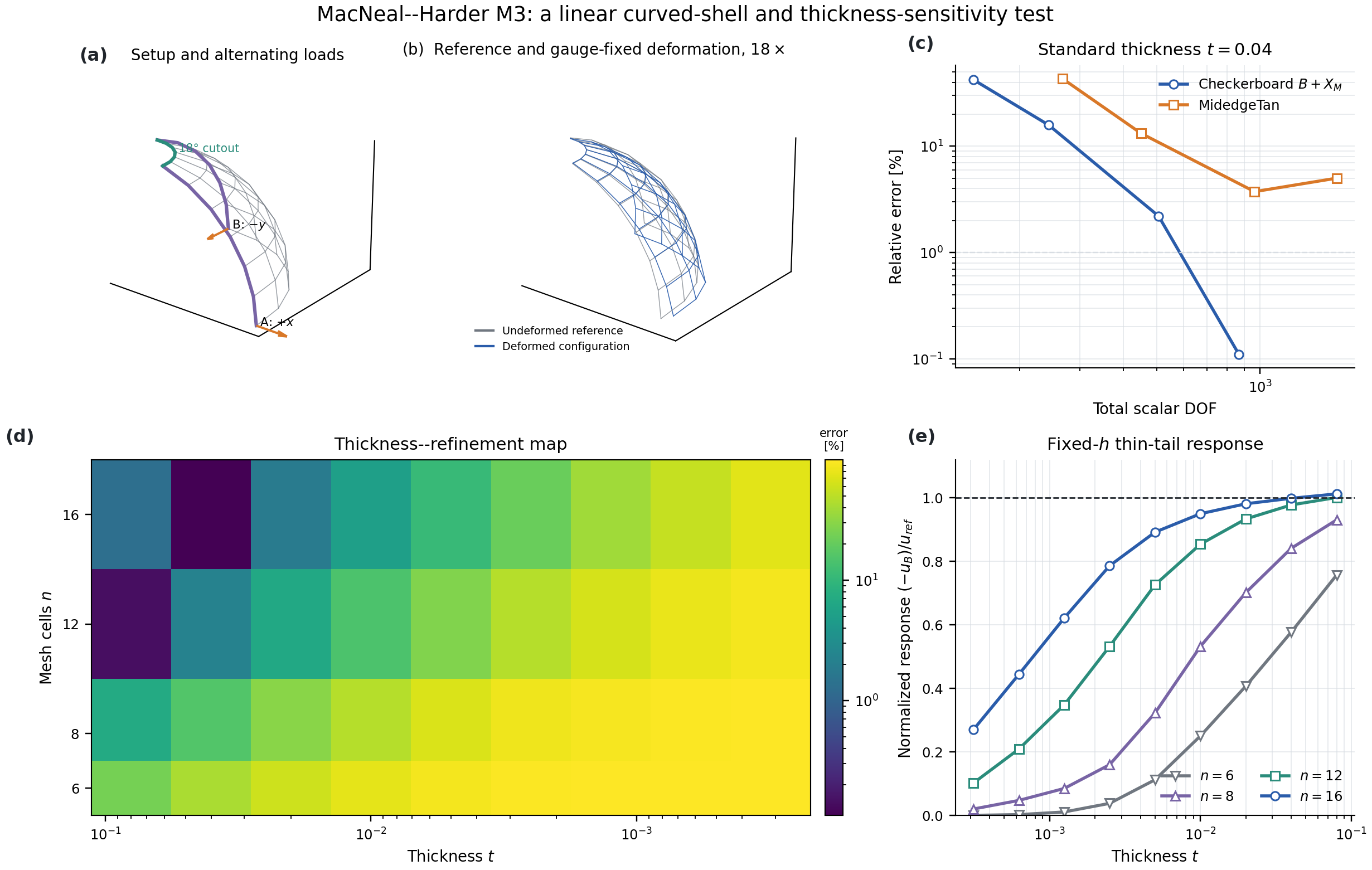}
 \caption{MacNeal--Harder M3 as a linear curved-shell and thickness-sensitivity benchmark. (a) $A=(R,0,0)$ carries a $+x$ outward load and $B=(0,R,0)$ a $-y$ inward load; the quarter hemisphere, $18^\circ$ cutout, and symmetry edges are also shown. (b) Medium-gray undeformed surface and blue gauge-fixed reconstructed deformation amplified by 18; the B end moves toward the sphere interior and the reported observable is $-u_B$. (c) Error at $t=0.04$ relative to the standard reference versus total scalar DOF. (d) Thickness--refinement error map. (e) Fixed-$h$ normalized response. Reconstruction is used only for visualization; mechanical quantities are evaluated on the physical Checkerboard geometry.}
 \label{fig:m3}
\end{figure}

\subsection{Geometrically Nonlinear Benchmarks}\label{sec:nonlinear}

The linear benchmarks probe only the reference tangent. Finite configurations additionally test the current frame, normal transport, and large-rotation updates. The three nonlinear cases are chosen to separate progressively more complex deformation mechanisms: the cantilever isolates extreme rotation under nearly pure bending; the hemisphere carries the classical M3 geometry into a finite-deformation full load path; and the open cylinder adds a localized point load, strong curvature gradients, and open-section ovalization. Together they cover large rotation, nonlinear membrane--bending coupling on a curved reference, and localized strong-curvature response.

All three nonlinear problems recover the reference tangent response as $\lambda\to0$ and satisfy component-energy objectivity and internal-residual covariance at representative finite-deformation states. Every load is derived from an explicit external potential. The raw gauge is removed by quotient-based reduction rather than a penalty. The cantilever, hemisphere, and open cylinder are all solved with adaptive load-control Newton iterations and require neither arc-length continuation nor benchmark-specific tuning.

\subsubsection{End-Moment Cantilever}\label{sec:cantilever}

The rectangular cantilever has $E=1.2\times10^6$, $\nu=0$, $L=12$, width $=1$, and $t=0.1$, with maximum end moment
\begin{equation}
 M_{\max}=\frac{2\pi EI}{L}=\frac{50\pi}{3}.
\end{equation}
Because the position-only state has no independent rotational DOF, the end moment is applied through a continuous unwrapped rotation potential of a gauge-invariant physical-midpoint end frame. The resulting positional forces have zero resultant and the prescribed resultant moment. The analytical reference is the constant-curvature circular arc \cite{szeEtAl2004,neunteufelSchoberl2019}.

Figure~\ref{fig:nonlinear-ch}(a) marks the applied end couple with a large-radius thin circular arrow; the positive sign rotates the end tangent from $+x$ toward $+z$. Meshes 16$\times$4, 24$\times$4, and 32$\times$4 all pass continuously through $90^\circ$ and $180^\circ$ and reach full-circle-level bending. Figure~\ref{fig:nonlinear-ch}(b) shows centerline configurations at 25/50/75/100\% load. Figure~\ref{fig:nonlinear-ch}(c) compares the complete tip path with the constant-curvature solution, and Fig.~\ref{fig:nonlinear-ch}(d) shows RMS and maximum path-error refinement. The fine mesh reaches a final turning angle of $359.841^\circ$, with RMS and maximum path errors of 2.442\% and 2.792\%, respectively. The final membrane-energy fraction is 0.432\% and decreases under refinement; the $X_M$ and $X_W^{\rm rel}$ energies are approximately $4.6\times10^{-64}$ and $1.8\times10^{-30}$, respectively, and are numerically inactive on this pure-bending branch. This case therefore probes finite rotation, objectivity, and a position-only applied couple, but not curved-reference transport in isolation.

\subsubsection{Alternating-Force Hemisphere}\label{sec:nonlinear-hemisphere}

This problem uses the same $R=10$, $t=0.04$, $E=6.825\times10^7$, $\nu=0.3$, and $18^\circ$ cutout as M3, but solves the full finite-configuration equilibrium path. In the full model, four alternating radial forces each reach magnitude 100; in the quarter model, each symmetry point carries half of the corresponding dead-load resultant. The two radial edges satisfy reflection symmetry. The observables are $u_A=\Delta x_A$ and $-u_B=-\Delta y_B$; the local notation used in the literature satisfies $V_A\equiv u_A$ and $-U_B\equiv-u_B$. Figure~\ref{fig:nonlinear-ch}(e) shows the medium-gray reference, blue final configuration, and restrained orange load arrows. Point A moves outward from the sphere and B moves inward. Meshes $n=6,8,12,16$ all complete the full path $P/P_{\max}\in[0,1]$. The reference path is digitized from the published load--deflection curves \cite{szeEtAl2004}; the final B value has an independent tabulated cross-check of 3.8796 \cite{neunteufelSchoberl2019}. The digitization uncertainty in displacement is approximately 0.03.

Figure~\ref{fig:nonlinear-ch}(f) gives the four-mesh paths of $u_A$, and Fig.~\ref{fig:nonlinear-ch}(g) the corresponding paths of $-u_B$. On the $n=16$ mesh, the two final errors are 2.759\% and 0.149\%, the joint RMS path error is 1.736\%, and the maximum normalized path error is 1.859\%. Both complete curves and the near-terminal data approach the same reference branch under refinement. The maximum $X_M/E_{\rm int}$ is $1.62\times10^{-7}$, while the maximum $X_W^{\rm rel}/E_{\rm int}$ is 0.322\%; the latter decreases from 1.552\% at $n=6$. At all representative load levels, the soft tangent modes correspond to physical ovalization/pinching rather than lattice-scale checkerboard patterns. This problem therefore simultaneously tests curved-reference transport, nonlinear membrane--bending coupling, and a complete load path.

\begin{figure}[t]
 \centering
 \includegraphics[width=\textwidth]{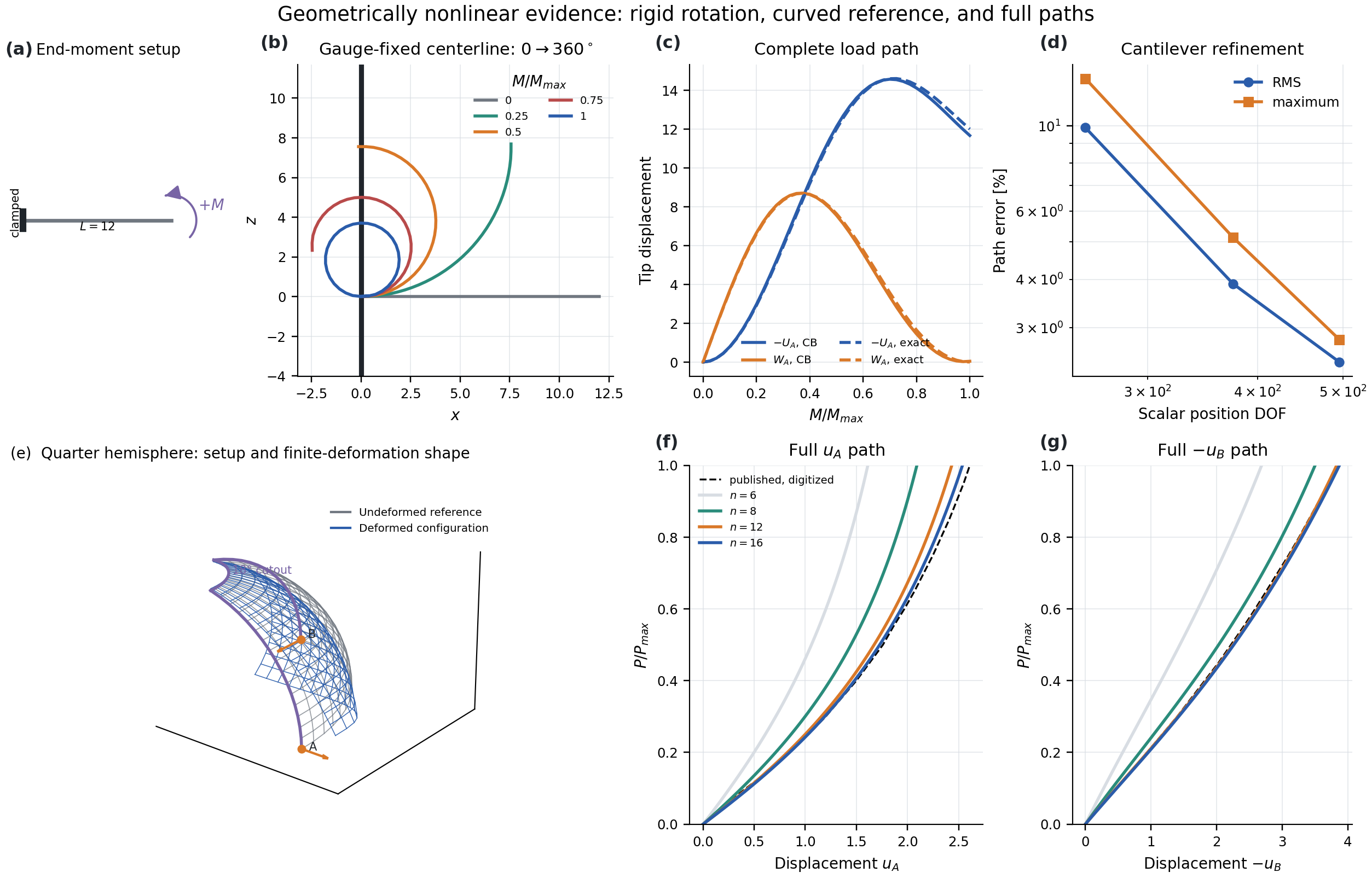}
 \caption{Geometrically nonlinear validation I. Top row, end-moment cantilever: (a) clamp and positive end couple rotating the tangent from $+x$ toward $+z$; (b) gauge-fixed centerline sequence at $M/M_{\max}=0,0.25,0.5,0.75,1$; (c) complete tip $-U_A/W_A$ path and the exact constant-curvature solution; (d) RMS and maximum path-error refinement. Bottom row, nonlinear hemisphere: (e) $+x$ outward load at $A=(R,0,0)$ and $-y$ inward load at $B=(0,R,0)$, together with the medium-gray reference and blue final geometry; A moves outward and B inward, and the orange arrows indicate external loads only. (f,g) Full $u_A$ and $-u_B$ paths for meshes 6/8/12/16; dashed curves are digitized references in the literature notation $V_A,-U_B$. Reconstruction is used only for visualization; mechanical quantities are evaluated on the physical Checkerboard geometry.}
 \label{fig:nonlinear-ch}
\end{figure}
\FloatBarrier

\subsubsection{Open-Ended Cylinder Pull-Out}\label{sec:open-cylinder}

The first two nonlinear examples are dominated by global large rotation and symmetric curved-surface deformation. They do not simultaneously impose a localized point load, a free open edge, and a strong curvature gradient. The open-ended cylinder pull-out benchmark adds precisely these features. A radial force pulls the midspan section outward and strongly ovalizes the shell, while the response at the axially free end reverses after an intermediate load level. The test therefore probes whether both structured directions transmit localized deformation coherently and whether lattice-scale checkerboard artifacts appear under strong local curvature.

We use the benchmark of Sze, Liu, and Lo \cite{szeEtAl2004}, cross-checked against rotation-free IGA, solid-shell, and IGA reproductions \cite{duongEtAl2017,wangEtAl2017,nguyenEtAl2015}:
\begin{equation}
 L=10.35,\quad R=4.953,\quad t=0.094,\quad
 E=10.5\times10^6,\quad \nu=0.3125.
\end{equation}
The one-eighth domain is $x\in[0,L/2]$, $\theta\in[0,\pi/2]$, with parametrization $X=(x,R\cos\theta,R\sin\theta)$. In Fig.~\ref{fig:open-cylinder}(a), a small inset highlights the computational domain on a translucent full cylinder, while the main setup isolates the three reflection boundaries $x=0$, $\theta=0$, and $\theta=\pi/2$, together with the traction- and moment-free open end at $x=L/2$. In the full model, each of the two opposite radial dead forces reaches magnitude 40000; the reduced domain carries $P/4=10000$ in the $+z$ direction at $A=(0,0,R)$. The plotted measurement directions match the definitions used in the text: $w_A=\Delta z_A$, $-u_B=-\Delta y_B$ at the midspan point $B=(0,R,0)$, and $-u_C=-\Delta y_C$ at the open-end point $C=(L/2,R,0)$. The A/B/C reference paths are digitized from literature vector graphics and are consistently identified as digitized references.

Meshes 8$\times$8, 12$\times$12, and 16$\times$16 all reach $P/P_{\max}=1$ using the same adaptive Newton/load-control strategy, requiring 45, 45, and 46 accepted load steps, respectively, with no arc-length continuation. Figure~\ref{fig:open-cylinder}(b) shows the full A/B/C paths on the 16$\times$16 mesh. The fine-mesh final errors are 8.094\%, 5.316\%, and 1.861\%, with corresponding RMS path errors of 7.075\%, 4.366\%, and 2.319\%. Figure~\ref{fig:open-cylinder}(c) enlarges the reversal at point C. The combined RMS error decreases from 13.648\% to 7.340\% and 4.483\%, and all three meshes capture the displacement reversal near $P/P_{\max}\simeq0.5$, with the peak moving toward the reference path under refinement.

Along the full path, the maximum meaningful $E_{X_M}/E_{\rm int}$ is 0.001236\%, and the maximum $E_{X_W}/E_{\rm int}$ is 0.202622\%. No normal flips, orientation reversals, or degenerate B/W sites occur. The softest tangent modes are classified as physical ovalization modes, and their alternating scores decrease under refinement. Figure~\ref{fig:open-cylinder}(d) overlays the medium-gray reference and the colored final geometry, showing axial/circumferential bending together with ovalization of the open and midspan sections, without local alternating folds. This benchmark therefore provides evidence for localized nonlinear response while retaining an explicit finite-mesh error assessment on the 16$\times$16 grid.

\begin{figure}[t]
 \centering
 \includegraphics[width=\textwidth]{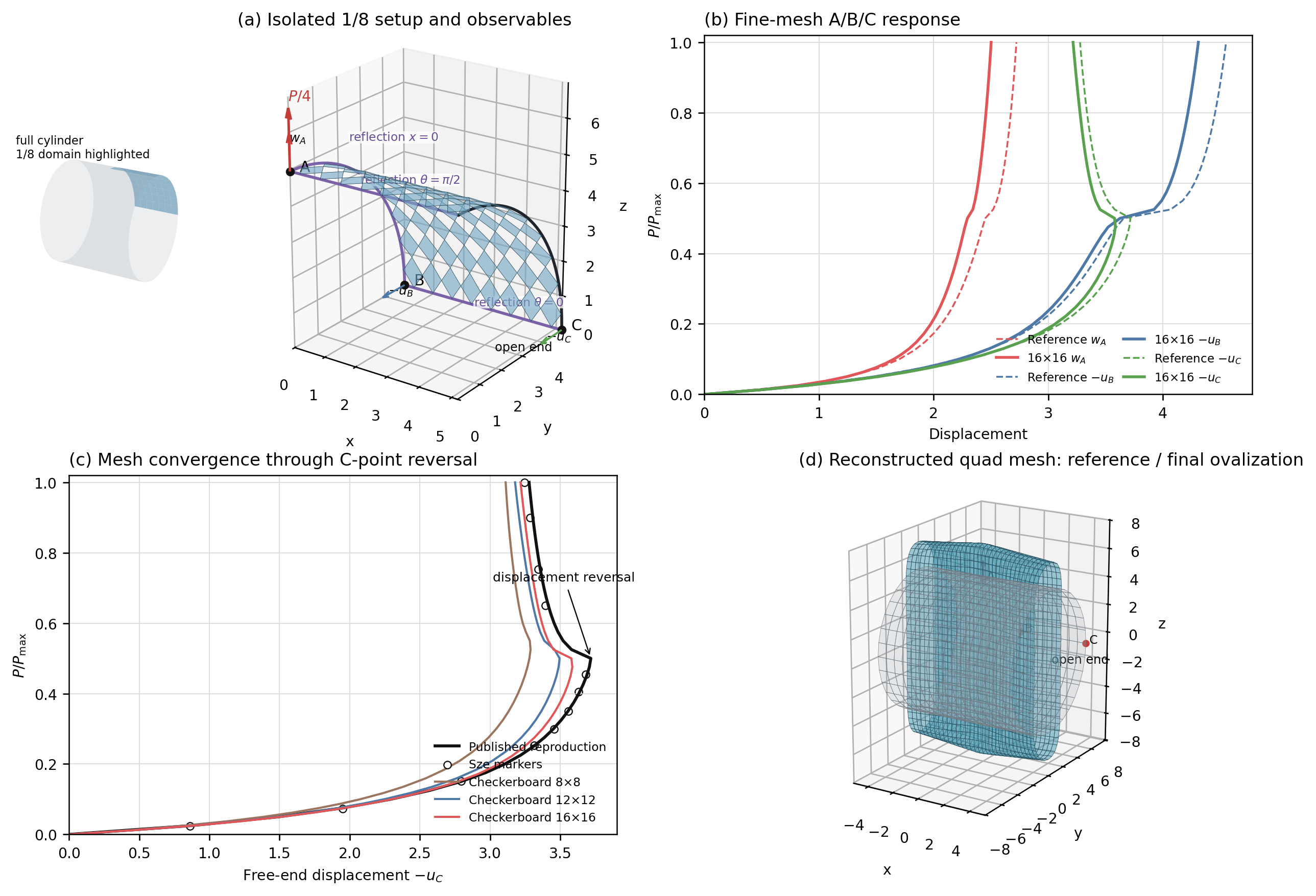}
 \caption{Geometrically nonlinear validation II: open-ended cylindrical-shell pull-out. (a) A small inset highlights the one-eighth computational domain on a translucent full cylinder; the main setup isolates the three reflection boundaries, the free open end, the reduced dead load $P/4$, and the measurement directions $w_A,-u_B,-u_C$. (b) Complete A/B/C paths on the 16$\times$16 mesh and the digitized references. (c) Refinement of the displacement reversal at point C for 8$\times$8, 12$\times$12, and 16$\times$16 meshes. (d) Medium-gray undeformed cylinder and fine-mesh final ovalization mirrored from the physical B-face geometry, with load point and open end indicated.}
 \label{fig:open-cylinder}
\end{figure}
\FloatBarrier

Across the linear and finite-deformation tests, the same B/W midpoint geometry, positional state, and geometric update remain in use. The resulting benefits therefore extend beyond representation to the organization of membrane--bending constraints and to finite-configuration solution procedures.
\section{Discussion}\label{sec:discussion}

\subsection{Advantages of Checkerboard Geometry for Shell Discretization}

The Varignon midpoint surface first provides a uniquely defined local geometry. A raw quadrilateral may become strongly warped, yet its B face remains exactly planar, and its tangent frame, normal, and first fundamental form are obtained directly from nodal positions. The W-centered second form compares only neighboring B-face normals, so membrane and bending share the same physical surface. In contrast to shell formulations that carry independent directors, rotations, or edge states, no auxiliary rotational variables must be updated or kept consistent during finite rotation; the state remains $3V$ positional scalars throughout.

This compact state translates into a measurable accuracy-per-state advantage in the standard benchmarks. For the simply supported plate at $n=16$, Checkerboard uses 867 positional scalars to obtain an error of 1.488\%, whereas MidedgeTan uses 1667 scalar states on the same mesh and gives an error of 1.248\%. The two errors are already comparable while the state size is nearly halved. In the more demanding M3 hemisphere, Checkerboard reaches 2.20\% error with 507 scalar states at $n=12$, whereas MidedgeTan uses 963 states on the same mesh and gives 3.74\% error. The Scordelis--Lo roof likewise shows stable convergence while retaining roughly half as many scalar states. Position-only kinematics therefore do more than reduce the number of variable types: on the plate and curved-shell problems considered here, they deliver a high accuracy per scalar state. The M3 result further shows that this compactness does not require sacrificing the low-order flexibility needed for a locking-sensitive near-isometric response.

The quotient construction and targeted completion preserve this economy. The midpoint quotient removes a raw gauge that leaves the physical midsurface unchanged, rather than fixing redundant directions through material penalties. $X_M$ adds only a rank-one membrane channel, and its energy contribution in the standard M3 response is essentially at machine precision. The fine-mesh displacement shifts induced by $X_W^{\rm rel}$ remain below $0.3\%$ in the plate, Scordelis--Lo, and M3 tests and decay under refinement. The missing curvature information can therefore be restored without materially altering the normal bending response. In this sense, the added compatibility observations act primarily on directions genuinely absent from the Checkerboard lattice rather than stabilizing the full space through broad stiffening.

The finite-deformation results reveal another advantage of the unified position-only geometry. The end-moment cantilever reaches a turning angle of $359.841^\circ$ continuously, with a fine-mesh tip-path RMS error of 2.442\%. The nonlinear hemisphere has a joint RMS path error of 1.736\% while preserving the classical outward motion at A and inward motion at B. The 16$\times$16 open-cylinder model captures the full ovalization pattern and displacement reversal. In all three problems the geometry is updated directly from the current positions, without switching kinematic descriptions for large rotation. Exact local planarity, a shared membrane--bending carrier, a low scalar-state count, and low-rank completion of the identified blind modes together form the principal numerical advantages of the Checkerboard shell discretization.

\subsection{Structured Implementation and Extensions}

A structured quadrilateral Checkerboard grid expresses the B/W staggering, the $M$ character, and the boundary moment reconstruction through fixed stencils. This regularity is an important reason the formulation can retain a compact positional state and explicit geometric structure. Extensions to multi-patch or general quadrilateral topologies can preserve the central construction---the midpoint surface as the physical carrier and neighboring planar-patch normals as the bending descriptor---while redefining local B/W coupling at patch interfaces, extraordinary vertices, and irregular neighborhoods. Such extensions modify the stencil and topological organization rather than the underlying geometric modeling principle.

The thickness--mesh tests show that refinement moves the residual membrane influence of curved near-isometries into progressively thinner regimes. Further improvement for extremely thin shells can therefore build on the existing blind-mode completion through higher-order near-isometry approximation and more local membrane observations. Likewise, the global rank-one structure of $X_M$ suggests equivalent localized representations suitable for domain decomposition or adaptive computation, while the surface-polar transport of $X_W^{\rm rel}$ can be extended naturally to multi-patch interfaces and finite-deformation contact. The common principle is to derive geometry from the physical midpoint surface and to add constraints only where a specific loss of variational information has been identified.

\clearpage
\section{Conclusions}\label{sec:conclusion}

Checkerboard edge-midpoint geometry provides a natural local carrier for position-only thin-shell discretization. Once the connected edge-midpoint surface is taken as the physical discrete midsurface, the strictly planar Varignon B faces directly define the tangent frame, normal, and first fundamental form, while variations of neighboring B-face normals define the second fundamental form at staggered W sites. Membrane and bending geometry are therefore determined, for arbitrary finite configurations, by the same nodal positions and the same midpoint surface, without additional director or rotational degrees of freedom.

Geometric consistency alone is not sufficient for a reliable shell energy; the discrete fundamental forms must also observe genuine physical deformation. The raw nodal representation has an exact checkerboard gauge for a fixed midpoint surface. In the physical quotient space, the B metric and symmetric W curvature additionally miss one membrane direction and two curvature directions, respectively. Identifying these three blind modes makes it possible to design the completion directly around the missing information. $X_M$ restores the unique membrane defect through a rank-one $M$-character observation and uses $\beta_M=8\mu$ to match the isotropic trace-free shear channel. $X_W^{\mathrm{rel}}$ restores the two curvature defects through an alternating B-normal component and surface-polar transport, with coefficient $\beta_W^0=\mu\gamma_W^0$ obtained by exact integration of an auxiliary-Q1 normal-gradient energy. The two terms enter at the $O(t)$ membrane and $O(t^3)$ bending scales, respectively, yielding a finite-configuration Koiter-type position-only energy.

The analytical and numerical results support the construction from complementary directions. Quotient and lattice-kernel analyses rigorously separate representation redundancy from physical blind modes; objectivity and reference-state analysis ensure that neither rigid motions nor gauge changes create artificial energy; and the $O(h^2)$ near-isometry approximation on aligned generalized cylinders shows that the minimal membrane completion preserves a natural curved-shell bending branch. The second-form tests on four smooth surfaces exhibit second-order geometric convergence. Mechanism tests show that $X_M$ is essentially inactive on normal load-bearing response and that $X_W^{\mathrm{rel}}$ introduces only a small finite-mesh correction that decays under refinement. The plate, Scordelis--Lo, M3, and three finite-deformation benchmarks span flat bending, membrane--bending coupling under initial curvature, thickness-sensitive response, large rotation, nonlinear pinching, and localized ovalization.

The central outcome is therefore not merely another discrete curvature measure, but a complete route from Checkerboard geometry to thin-shell mechanics: the midpoint surface defines the physical state, the quotient removes representation redundancy, variational observability of the fundamental forms identifies genuine defects, and scale-compatible completion restores precisely the missing information. This framework preserves the strictly planar local patches and purely positional state of Checkerboard geometry while unifying discrete geometry, variational observability, and finite-deformation shell energy on the same midpoint surface. Checkerboard geometry thereby becomes not only a design representation, but also a direct position-only discretization for membrane--bending coupled and large-rotation shell mechanics.

\clearpage

\section*{Appendix}
\addcontentsline{toc}{section}{Appendix}
\appendix
\section{Midpoint Quotient and Gauge Identities}\label{app:quotient}

This appendix supplies the representation-level proof omitted from Section~\ref{sec:quotient} and clarifies the distinction between the compatible midpoint image and an arbitrary edge field. Let $V$ and $E$ denote the vertex and edge sets of the raw grid. For a scalar raw field $x\in\R^{|V|}$, define the undirected edge-sum operator, without the factor one-half, by
\begin{equation}
 (Sx)_{uv}=x_u+x_v,\qquad (u,v)\in E.
\end{equation}
The three-dimensional midpoint map can then be written as $P=\tfrac12 S\otimes I_3$. The statements below require only that the raw graph be connected and bipartite; rectangular regularity is not needed.

\begin{lemma}[Scalar edge-sum kernel]\label{lem:edge-sum}
If $G=(V,E)$ is connected and bipartite, with bipartite coloring $p_v\in\{+1,-1\}$, then
\begin{equation}
 \ker S=\operatorname{span}\{p\}.
\end{equation}
\end{lemma}

\begin{proof}
Take any $x\in\ker S$. On every edge $(u,v)$, $x_v=-x_u$. Fix a root vertex $v_0$. Propagating along a path of length $k$ gives $x_v=(-1)^kx_{v_0}$. In a bipartite graph, any two paths from $v_0$ to $v$ have the same parity; otherwise their concatenation would contain an odd cycle. Hence the value is path independent. Connectivity implies that all vertex values are determined by the single scalar $x_{v_0}$, with signs given by the bipartite coloring. Conversely, substitution directly gives $S(cp)=0$.
\end{proof}

Applying \cref{lem:edge-sum} independently to the three Cartesian components gives $\ker P=\{pA:A\in\R^3\}$ in \cref{thm:gauge}. The first isomorphism theorem then yields
\begin{equation}
 \Qraw/\ker P\simeq\im P.
\end{equation}
The right-hand side must be understood as the \emph{compatible midpoint image}. An arbitrary assignment of edge vectors does not generally lie in the image of $P$, because alternating edge sums around every even cycle must satisfy a closure relation. This is why raw positions are convenient parameters for a physical midpoint surface: midpoint compatibility is satisfied automatically, while only a fixed three-dimensional gauge remains.

The roles of the physical boundary operator $C_{\rm phys}$ and the gauge slice $C_g$ should be kept distinct. The former imposes $C_{\rm phys}Pq=c$, or an equivalent midpoint extrapolation condition, whereas the latter selects a single representative from the affine class $q+pA$. If the columns of $T$ span
\begin{equation}
 \ker\begin{bmatrix}C_{\rm phys}P\\C_g\end{bmatrix},
\end{equation}
then the reduced solve $q=q_b+Tz$ satisfies both the physical boundary conditions and the uniqueness of the raw representative. Changing the admissible $C_g$ modifies only the raw display, not $Pq$, the internal energy, the external work, or any gauge-invariant observable. The energy and residual remain unchanged under a finite-amplitude perturbation $pA$.
\section{Flat First- and Second-Form Kernel Proofs}\label{app:kernels}

\subsection{In-Plane Quotient Defect of the B Metric}

On a unit regular flat grid, choose equivalent local coordinates in which the B diagonal frame is $T_B^0=I_2$. For an in-plane raw displacement $v_{ij}\in\R^2$, the increment of the B tangents can be written as a $2\times2$ matrix $F_{ij}$, and
\begin{equation}
 D a_B[v]=F_{ij}+F_{ij}^T.
\end{equation}
Hence $Da_B[v]=0$ if and only if $F_{ij}=\omega_{ij}J$, where $J$ is the planar $90^\circ$ rotation. Shared raw nodes require the diagonal differences of neighboring B frames to arise from the same displacement field. Eliminating the nodal values gives the cell-microrotation recurrence
\begin{equation}
 \omega_{i+1,j+1}-\omega_{ij}=0,\qquad
 \omega_{i+1,j}-\omega_{i,j+1}=0.
\end{equation}
The first equation makes $\omega$ constant along one diagonal direction, while the second couples the values on the other diagonal by parity. On a connected rectangular patch, the general solution is
\begin{equation}
 \omega_{ij}=c_0+c_M(-1)^{i+j}.
\end{equation}
The $c_0$ component integrates to an ordinary rigid yaw. For $c_M=1$, choosing the representative with zero translation gives
\begin{equation}
 z_M(i,j)=h(-1)^{i+j}(i,-j).
\end{equation}
This is not a raw gauge. For example, the horizontal-edge midpoint variation contains
\begin{equation}
 \frac12[z_M(i,j)+z_M(i+1,j)]
 =-\frac h2(-1)^{i+j}(1,0),
\end{equation}
which is generically nonzero; nor can $z_M$ be represented by a single affine rigid field. Adding the two translations, rigid yaw, the two-dimensional raw gauge, and $z_M$ gives nullity $2+1+2+1=6$:
\begin{equation}
 \ker J_B^{\rm in}=\Rr^{\rm in}\oplus\Gg^{\rm in}
 \oplus\operatorname{span}\{z_M\}.
\end{equation}

The derivative $DX_M$ vanishes on the first three subspaces, while under the normalization used above
\begin{equation}
 |DX_M[z_M(\omega)]|^2=A_M^0\omega^2>0.
\end{equation}
Thus the joint kernel removes $z_M$ and only $z_M$. Since the quotient defect is one-dimensional, any additional linear observation capable of eliminating it must contribute at least one scalar row; the single row provided by $X_M$ attains this lower bound.

\subsection{Transverse Kernel of the Symmetric W Curvature}

Let $w_{ij}$ denote the transverse displacement of a flat reference configuration and set $p_{ij}=(-1)^{i+j}$. The linearized variation of a B normal consists of diagonal first differences, while $b_W$ takes their W-centered symmetric difference. The resulting stencil therefore forms a discrete Hessian on a diagonal graph that preserves raw parity. Introduce diagonal coordinates $\xi=i+j$ and $\eta=i-j$ separately on the even and odd sublattices. The pure second-difference equations and the mixed equation require each parity branch to be affine:
\begin{equation}
 w^{\pm}(\xi,\eta)=a_\pm+b_\pm\xi+c_\pm\eta.
\end{equation}
Rewriting the two parity-affine branches in $(i,j)$ coordinates gives the equivalent six-dimensional space
\begin{equation}
 \operatorname{span}\{1,i,j,p,ip,jp\}.
\end{equation}
The modes $1,i,j$ are a transverse translation and two infinitesimal rotations; $p$ leaves all edge midpoints fixed; and the edge averages of $ip,jp$ are nonzero, so they are physical curvature blind modes.

The same result can be interpreted from the root multiplicity of the Fourier symbol. The symbol of $b_W$ has second-order zeros at both $\Gamma=(0,0)$ and $M=(\pi,\pi)$. The constant and first jets at $\Gamma$ correspond to $1,i,j$, whereas those at $M$ correspond to $p,ip,jp$. The alternating sum of the four B normals, $X_W$, is nonzero on the first jets at $M$, while the symmetric curvature already observes the stripe points $(\pi,0)$ and $(0,\pi)$. Therefore
\begin{equation}
 \ker(D b_W,D X_W)=\operatorname{span}\{1,i,j,p\},
\end{equation}
which is exactly the rigid-plus-gauge target of the flat bending sector.
\section{Objectivity and Compatibility Details}\label{app:objectivity}

\subsection{Finite Gauge Invariance and Rigid Covariance of the B/W Geometry}

For the raw gauge $g_{ij}=p_{ij}A$, each B diagonal difference connects vertices of the same color, so the gauge increments at its two endpoints are identical and cancel under subtraction. Equivalently, all edge midpoints remain unchanged under $q\mapsto q+g$. Hence $T_B,N_B,T_W,D^hN,a_B$, and $b_W$ are all exactly gauge invariant. Under a superposed rigid motion $q\mapsto Qq+c$,
\begin{equation}
 T_B\mapsto QT_B,\quad N_B\mapsto QN_B,\quad
 T_W\mapsto QT_W,\quad D^hN\mapsto QD^hN,
\end{equation}
so $a_B$ and $b_W$ remain unchanged.

\subsection{Reference-Relative Objectivity of \texorpdfstring{$X_M$}{XM}}

The polar-normalized frame $\mathcal E_B=T_Ba_B^{-1/2}$ transforms as $\mathcal E_B\mapsto Q\mathcal E_B$ under a rigid rotation. The current/reference relative rotation
\begin{equation}
 R_B(q;q_0)=\mathcal E_B(q)\mathcal E_B(q_0)^T
\end{equation}
therefore becomes $QR_B$ when $Q$ is superposed on the current configuration only. In the neighboring-cell connection $C_B^\alpha=R_B^TR_{B+e_\alpha}$, the left factor $Q$ cancels. The Cayley axial coordinate and its projection onto the material/reference normal are therefore invariant, as are the plaquette mixed differences, the $M$-character sum, and ultimately $X_M$. If the current configuration is any rigid image of the reference, all $C_B^\alpha$ reduce to the reference-relative identity connection and $X_M=0$.

For an ordinary homogeneous affine membrane deformation, adjacent polar rotations are identical, so the connection differences vanish and $X_M$ does not measure the affine strain itself. The $M$-character normalization
\begin{equation}
 X_M=\frac{1}{4\sqrt{A_M^0}}
 \sum A_{ij}^{M,0}p_{ij}\mu_{ij}
\end{equation}
is chosen so that a unit alternating microrotation satisfies $|DX_M|^2=A_M^0\omega^2$. With this normalization, the fixed coefficient $\beta_M=8\mu$ matches a full-area trace-free unit shear channel. The coefficient is not tuned to a benchmark and is not claimed to be uniquely dictated by the continuum Gram law.

\subsection{Surface-Polar Transport of \texorpdfstring{$X_W^{\rm rel}$}{XW-rel}}

The raw alternating normal vector transforms as $X_W(q)\mapsto QX_W(q)$. With
\begin{equation}\label{eq:w-polar-app}
 \begin{aligned}
 n_W(q)&=\frac{t_1^W(q)\times t_2^W(q)}
 {\norm{t_1^W(q)\times t_2^W(q)}},\\
 E_W(q)&=[\,t_1^W(q),\;t_2^W(q),\;n_W(q)\,],
 \qquad E_W^0=E_W(q_0),\\
 F_W(q,q_0)&=E_W(q)[E_W(q_0)]^{-1},\\
 R_W(q,q_0)&=F_W(q,q_0)
 [F_W(q,q_0)^T F_W(q,q_0)]^{-1/2}\in SO(3),
 \end{aligned}
\end{equation}
the rotation is the orientation-preserving right polar factor of the W-centered current/reference deformation map. On the admissible set $F_W\in GL^+(3)$, a superposed rigid rotation $Q$ gives
\begin{equation}
 F_W(Qq+c,q_0)=QF_W(q,q_0),
 \qquad R_W(Qq+c,q_0)=Q R_W(q,q_0).
\end{equation}
Therefore
\begin{align}
 X_W^{\rm rel}(Qq+c;q_0)
 &=QX_W(q)-QR_W(q,q_0)X_W(q_0)\\
 &=QX_W^{\rm rel}(q;q_0),
\end{align}
and its Euclidean norm is invariant. If the direct difference $X_W(q)-X_W(q_0)$ were used instead, the second term would not rotate with the current rigid motion and the quantity would not be objective for a curved reference. Conversely, transporting each of the four B normals independently would also remove the alternating component itself. Using a single W surface-polar rotation preserves precisely the normal-compatibility information that must remain observable.

\subsection{Determination of the Membrane and Curvature Completion Coefficients}\label{app:completion-coefficients}

The coefficients of the two compatibility energies are obtained by matching the prescribed discrete observations to the corresponding isotropic shell-energy channels; no benchmark fitting is used. On the membrane side, the unique missing $M$-point quotient channel is normalized to continuous trace-free shear. On the curvature side, the alternating B-normal component on a W patch is interpreted as an auxiliary-Q1 normal field and its metric-gradient energy is integrated exactly.

\paragraph{Coefficient $\beta_M$ for $X_M$.}
In local orthonormal coordinates of the B frame, consider the unit trace-free shear tensor
\begin{equation}
 S=e_1\otimes e_2+e_2\otimes e_1,
 \qquad \tr S=0,\qquad \tr(S^2)=2.
\end{equation}
Let $\omega$ be the amplitude of the missing $M$-point microrotation. With the area normalization in \cref{eq:xm}, the pure $M$-character channel satisfies
\begin{equation}
 |DX_M|^2=A_M^0\omega^2.
\end{equation}
The continuum Gram law does not assign a modulus to the skew microrotation itself, because the $z_M$ direction satisfies $Da_B[z_M]=0$. To set an a priori constitutive scale for the missing scalar channel without benchmark fitting, we adopt the normalization convention
\begin{equation}
 Z_M^{\rm cal}=2\omega S.
\end{equation}
Here $Z_M^{\rm cal}$ is a latent trace-free shear calibration channel associated with the amplitude $\omega$; it is not the actual B-metric variation of $z_M$. In the orthonormal reference frame, the plane-stress Koiter contraction \cref{eq:qa} gives
\begin{equation}
 Q_I(Z_M^{\rm cal})
 =\mu\tr[(Z_M^{\rm cal})^2]
 =\mu(4\omega^2)\tr(S^2)
 =8\mu\omega^2.
\end{equation}
Matching the added scalar observation to the same unit-area material shear channel requires
\begin{equation}
 \beta_M A_M^0\omega^2
 =A_M^0Q_I(Z_M^{\rm cal}).
\end{equation}
Canceling $A_M^0\omega^2$ yields
\begin{equation}
 \beta_M=8\mu=\frac{4E}{1+\nu}.
\end{equation}
Thus $\beta_M$ fixes the constitutive normalization of $X_M$ through the isotropic shear modulus and the discrete normalization in \cref{eq:xm}. This is an a priori constitutive extension convention paired with the chosen $X_M$ normalization; it is not a unique coefficient implied by the continuum Gram law, and it is independent of load, thickness, or benchmark fitting.

\paragraph{Coefficient $\beta_W^0$ for $X_W^{\mathrm{rel}}$.}
Let the parameter domain of a complete reference W patch be
\begin{equation}
 \Omega_W=[-h_1/2,h_1/2]\times[-h_2/2,h_2/2],
\end{equation}
with local metric $g=g_W^0$. The alternating combination of the four corner B normals is
\begin{equation}
 X_W=N_{NE}-N_{NW}+N_{SW}-N_{SE}.
\end{equation}
Within a Q1 interpolation, the pure bilinear normal component associated with this alternating corner mode can be written as
\begin{equation}
 n_M(x,y)=\frac{xy}{h_1h_2}X_W.
\end{equation}
Its alternating sum at the four corners is exactly $X_W$, and its parametric derivatives are
\begin{equation}
 \partial_1 n_M=\frac{y}{h_1h_2}X_W,
 \qquad
 \partial_2 n_M=\frac{x}{h_1h_2}X_W.
\end{equation}
Substituting into the normal-gradient quadratic form on the reference metric gives
\begin{equation}
 I_W=
 \int_{\Omega_W}\sqrt{\det g}\,
 g^{\alpha\beta}
 \partial_\alpha n_M\cdot\partial_\beta n_M\,dx\,dy.
\end{equation}
Because $\Omega_W$ is centered symmetrically in $x$ and $y$, the $g^{12}$ cross term proportional to $xy$ integrates to zero, while
\begin{equation}
 \int_{\Omega_W}\frac{y^2}{h_1^2h_2^2}\,dx\,dy
 =\frac{1}{12}\frac{h_2}{h_1},
 \qquad
 \int_{\Omega_W}\frac{x^2}{h_1^2h_2^2}\,dx\,dy
 =\frac{1}{12}\frac{h_1}{h_2}.
\end{equation}
Hence
\begin{equation}
 I_W=
 \frac{\sqrt{\det g_W^0}}{12}
 \left[
 (g_W^0)^{11}\frac{h_2}{h_1}
 +(g_W^0)^{22}\frac{h_1}{h_2}
 \right]\norm{X_W}^2
 =\gamma_W^0\norm{X_W}^2.
\end{equation}
The curvature-compatibility energy inherits the shear modulus $\mu$ from the Koiter bending sector. Thus the coefficient multiplying the squared term in \cref{eq:xw-energy} is
\begin{equation}
 \beta_W^0=\mu\gamma_W^0
 =\frac{\mu\sqrt{\det g_W^0}}{12}
 \left[
 (g_W^0)^{11}\frac{h_2}{h_1}
 +(g_W^0)^{22}\frac{h_1}{h_2}
 \right].
\end{equation}
For an orthogonal square grid, $g_W^0=I$ and $h_1=h_2$, so $\gamma_W^0=1/6$ and $\beta_W^0=\mu/6$. Finite rotation replaces $X_W$ by the objective $X_W^{\rm rel}$ while leaving the reference-patch metric and its integration coefficient unchanged.
\section{Aligned Generalized-Cylinder Approximation}\label{app:cylinder}

This appendix gives the construction underlying \cref{thm:cylinder}. Consider
\begin{equation}
 r(x,z)=c(x)+ze_z,\qquad
 c'(x)\cdot e_z=0,\qquad |c'(x)|=1,\qquad
 e_\theta(x)=c'(x),\quad n(x)=e_\theta(x)\times e_z,
\end{equation}
where $c\in C^4$ has bounded signed curvature and the sign convention is
\begin{equation}
 e_\theta'(x)=-\kappa(x)n(x),\qquad n'(x)=\kappa(x)e_\theta(x).
\end{equation}
The grid nodes are $(x_i,z_j)$, with the $x_i$ uniformly sampled in arc length. Up to a constant rigid translation in the generator direction, let
\begin{equation}
 v(x,z)=u(x)e_\theta(x)+w(x)n(x)+c_ze_z,\qquad
 u'(x)+\kappa(x)w(x)=0,
\end{equation}
where $u,w\in C^3$ and $c_z$ is constant. This is a row-independent infinitesimal isometry of the continuous first fundamental form.

On the row-independent branch, the discrete B-metric zero condition reduces to one scalar recurrence for each circumferential chord. Define
\begin{equation}
 d_i=c(x_{i+1})-c(x_i),\qquad
 \widehat d_i=d_i/|d_i|,\qquad
 \delta v_i=v_{i+1}-v_i.
\end{equation}
The linearized chord-length condition is
\begin{equation}\label{eq:disc-cylinder-rec}
 \widehat d_i\cdot\delta v_i^C=0.
\end{equation}
Starting from the center node with $v_0^C=I_hv_0$, project each sampled increment $\delta I_hv_i$ onto $\widehat d_i^\perp$:
\begin{equation}
 \delta v_i^C=\delta I_hv_i
 -\widehat d_i(\widehat d_i\cdot\delta I_hv_i).
\end{equation}
Integrating this recurrence in both circumferential directions yields a row-independent field $v_h^C$ on the whole patch. By construction, \cref{eq:disc-cylinder-rec} holds exactly on every chord, while the axial and mixed B-metric variations vanish automatically by row independence; hence $J_Bv_h^C=0$.

For this row-independent branch, the infinitesimal polar rotations of the B cells are identical along each generator row. The increment of the relative connection in the generator direction is therefore zero, while the circumferential connection increment is independent of the row index. The two terms in \cref{eq:mixed-moment} consequently satisfy, pointwise,
\[
 D(\theta^x_{i,j+1}-\theta^x_{ij})[v_h^C]=0,
 \qquad
 D(\theta^y_{i+1,j}-\theta^y_{ij})[v_h^C]=0.
\]
Thus the linearized mixed moment vanishes on every complete plaquette, the $M$-character sum requires no boundary-parity cancellation, and
\begin{equation}
 DX_M[v_h^C]=0,\qquad v_h^C\in K_h^C=\ker(J_B,DX_M).
\end{equation}
If the patch carries essential boundary data, its trace must be compatible with the constructed $v_h^C$; a free patch requires no additional boundary condition.

It remains to estimate the correction. The arc-length chord expansion gives
\begin{equation}
 \widehat d_i=e_\theta(x_{i+1/2})+O(h^2),\qquad
 \delta I_hv_i=h\partial_xv(x_{i+1/2})+O(h^3).
\end{equation}
The continuous isometry condition implies $e_\theta\cdot\partial_xv=0$, and therefore
\begin{equation}
 \widehat d_i\cdot\delta I_hv_i=O(h^3).
\end{equation}
Each correction step is $O(h^3)$; accumulation over $O(h^{-1})$ circumferential intervals gives an $O(h^2)$ nodal correction, while its first differences remain $O(h^3)$. Substitution into the trapezoidal discrete norms yields
\begin{equation}
 \norm{I_hv-v_h^C}_{0,h}+h|I_hv-v_h^C|_{1,h}\le Ch^2,
\end{equation}
where $C$ depends only on the relevant Sobolev/$C^k$ bounds of $c$ and $v$, the domain length, and mesh shape regularity, but not on $h$. The physical-midpoint norm is a bounded average of neighboring raw values, so it obeys the same-order bound and exactly annihilates the raw gauge. This completes the proof.
\section{Boundary Closure}\label{app:boundary}

\subsection{Moment Reconstruction}

At a boundary located at $x=0$, let the three available interior samples be taken at $x=h/2$, $3h/2$, and $5h/2$. Quadratic moment reconstruction gives
\begin{equation}
 f(0)=\frac{15}{8}f(h/2)-\frac54 f(3h/2)+\frac38 f(5h/2)+O(h^3),
\end{equation}
and
\begin{equation}
 f'(0)=\frac{-2f(h/2)+3f(3h/2)-f(5h/2)}{h}+O(h^2).
\end{equation}
The formulas are exact for the value and derivative moments of polynomials up to degree two. Along two-dimensional edges and corners, the closure of $T_W,N_B,D^hN$ uses tensor products of this one-dimensional rule. The dual quadrature fractions for interior, edge, and corner sites are $1,1/2,1/4$, respectively. The resulting boundary geometry and full-domain integral retain second-order local consistency without introducing mirror ghost data that would implicitly impose a mechanical Neumann condition. The $X_W^{\rm rel}$ energy acts on complete interior W patches according to the definition of its auxiliary-Q1 coefficient, while $b_W$ itself is extended over the full finite domain by the moment closure above. These two treatments correspond, respectively, to geometric reconstruction of the curvature field and to integration of the high-frequency compatibility energy.
\section{Numerical Implementation and Additional Results}\label{app:numerical}

\subsection{Nonlinear Solution Procedure and Reference Data}

Linear problems are solved using the consistent reference-state Hessian. For nonlinear problems, both residual and tangent are obtained by automatic differentiation of the same scalar potential in \cref{eq:potential}. The solver uses Newton corrections, an energy-aware line search, adaptive load increments, rollback, and consistent convergence tolerances. None of the three complete nonlinear benchmarks requires arc-length continuation. In the $10^{-4}$, $10^{-3}$, and $10^{-2}$ load ranges, or their equivalent small-load regimes, $u(\lambda)/\lambda$ converges to the reference tangent response.

The cantilever reference is the analytical constant-curvature path. The complete hemisphere and open-cylinder paths are reproducibly digitized from verified vector graphics in the literature and are explicitly identified in the text and captions as digitized references. Independent tabulated endpoints are used only as cross-checks; the digitized curves are not represented as exact tabulated data. Path errors use RMS and maximum absolute measures normalized by the reference path norm or the maximum meaningful displacement, avoiding divergent pointwise relative errors near zero load.

\subsection{Spectral Properties and Finite-Mesh Contribution of the Curvature Completion}

For a free sphere patch, the numerical nullities of the B-only, $B+X_M$, and full B/W observations are 51, 50, and 9, respectively; on the active $n=16$ M3 space they are 10, 9, and 0. Thus $X_M$ changes only one curved active nullity, whereas a full B/W observation strongly compresses the original soft space. The fine-mesh displacement shifts induced by $X_W^{\rm rel}$ are 0.0136\%, 0.0210\%, and 0.2983\% for the plate, Scordelis--Lo, and M3 problems, respectively, and decrease further with refinement. These data reinforce the conclusion that $X_M$ acts primarily on the target quotient defect, while the finite-mesh influence of $X_W^{\rm rel}$ on standard load-bearing responses is small and decays under refinement.

\end{document}